\documentclass[12pt]{amsart}
\usepackage{amssymb}
\usepackage{mathtools}

\usepackage{tikz}
\usetikzlibrary{decorations.pathreplacing}
\usetikzlibrary{matrix}
\usetikzlibrary{calc}
 \usepackage{tikz-cd}
\tikzset{
  symbol/.style={
    draw=none,
    every to/.append style={
      edge node={node [sloped, allow upside down, auto=false]{$#1$}}}
  }
}
\usepackage[normalem]{ulem}  

\makeatletter
\newcommand*{\centerfloat}{%
  \parindent \z@
  \leftskip \z@ \@plus 1fil \@minus \textwidth
  \rightskip\leftskip
  \parfillskip \z@skip}
\makeatother

\usepackage[pdfencoding=auto,psdextra]{hyperref}
\usepackage{ytableau}
\usepackage[hmargin=1in,vmargin=1.25in]{geometry}
\usepackage{todonotes}
\allowdisplaybreaks

\definecolor{amethyst}{rgb}{0.6, 0.4, 0.8}
\definecolor{kellygreen}{rgb}{0.3, 0.73, 0.09}
\definecolor{americanrose}{rgb}{1.0, 0.01, 0.24}
\definecolor{teal}{rgb}{0, 0.5, 0.5}
\definecolor{lime}{rgb}{0.75, 1, 0}
\definecolor{darklime}{rgb}{0.25, .33, 0}

\makeatletter
\newcommand{\addresseshere}{%
  \enddoc@text\let\enddoc@text\relax
}
\makeatother

\newtheorem{thm}{Theorem}[subsection]
\newtheorem{lemma}[thm]{Lemma}
\newtheorem{prop}[thm]{Proposition}
\newtheorem{cor}[thm]{Corollary}
\newtheorem{conj}[thm]{Conjecture}

\theoremstyle{definition}
\newtheorem{defn}[thm]{Definition}
\newtheorem{open}[thm]{Open problem}

\theoremstyle{remark}

\newtheorem{remark}[thm]{Remark}

\newcommand{\PP}{{\mathsf{P}}}   
\newcommand{\Q}{{\mathsf{Q}}}   

\DeclareMathOperator{\area}{area}
\DeclareMathOperator{\Area}{Area}
\DeclareMathOperator{\content}{content}

\DeclareMathOperator{\FWPF}{FWPF}
\newcommand{\umnab}{\bar{\nabla}}   
\newcommand{\modnab}{\boldsymbol{\nabla}}
\newcommand{\umthe}{\bar{\Theta}}   
\newcommand{\modthe}{\boldsymbol{\Theta}}
\DeclareFontFamily{U}{mathx}{\hyphenchar\font45}
\DeclareFontShape{U}{mathx}{m}{n}{
      <5> <6> <7> <8> <9> <10>
      <10.95> <12> <14.4> <17.28> <20.74> <24.88>
      mathx10
      }{}
\DeclareSymbolFont{mathx}{U}{mathx}{m}{n}
\DeclareFontSubstitution{U}{mathx}{m}{n}
\DeclareMathAccent{\widecheck}{0}{mathx}{"71}
\newcommand{\Mnab}{\widecheck{\nabla}}

\DeclareMathOperator{\Par}{Par}

\renewcommand{\P}{\mathcal{P}}

\newcommand{\barwb}{\widecheck{w}}
\newcommand{\N}{\mathcal{N}}
\newcommand{\M}{\mathcal{M}}

\newcommand{\pairsr}{\ensuremath{\NN^\ell \times \Par}}
\newcommand{\tE}{ \check{\mathsf H} }   
\newcommand{\stE}{\check{J}}  
\newcommand{\flagh}{\ensuremath{\mathsf{h}}}
\newcommand{\nsC}{\ensuremath{\mathsf{C}}}
\newcommand{\nsM}{\ensuremath{\mathsf{M}}}
\newcommand{\alg}{\ensuremath{\nsM^\Q_{\alpha, k}}}

\newcommand{\sym}{{\mbox {Sym}}}

\newcommand{\pibold}{{\boldsymbol \pi }}
\newcommand{\Weyl}{\pibold }

\newcommand{\Grow}{\mathsf{G}}   
\newcommand{\Growpm}{\mathsf{G}^\pm}   
\newcommand{\Gcol}{\tilde{\mathsf{G}}}   
\newcommand{\Gcolpm}{\tilde{\mathsf{G}}^\pm}   

\newcommand{\Pisf}{{\mathsf \Pi }}

\newcommand{\NN}{{\mathbb N}}
\newcommand{\QQ}{{\mathbb Q}}

\newcommand{\ZZ}{{\mathbb Z}}
\renewcommand{\AA}{{\mathbb A}}

\newcommand{\Acal}{{\mathcal A}}

\newcommand{\Htild}{\tilde{H}}

\DeclareMathOperator{\dinv}{dinv}
\DeclareMathOperator{\inv}{inv}

\DeclareMathOperator{\pol}{pol}

\DeclareMathOperator{\FSSYT}{FSW}

\newcommand{\LD}{{\mathrm{LD}}}
\newcommand{\dcomp}{{\mathrm{dcomp}}}

\newcommand{\xx}{{X}}

\newcommand{\spa}{\hspace{.3mm}}

\newcommand{\e}{\underline{e}}

\newcommand\FrenchYD[1]{
    \foreach \rowlen [count=\y from 0] in {#1} {
        \foreach \x in {1,...,\rowlen} {
            \draw (\x-1,\y) rectangle (\x,\y+1);
        }
    }
}
\newcommand\PFmn[4]{
    \draw[help lines] (#1) grid ($(#1)+(#2,#3)$);
    \draw[thick, gray, dashed] (#1) -- ($(#1)+(#2,#3)$);
    \xdef\te{0} 
    \foreach \x/\y [count=\i] in {#4}{  
        \draw[line width=1.5pt, blue, line cap=round] ($(#1)+(\te,\i-1)$) -- ($(#1)+(\x,\i-1)$);
        \draw[line width=1.5pt, blue, line cap=round] ($(#1)+(\x,\i-1)$) -- ($(#1)+(\x,\i)$);
        \node at ($(#1)+(\x+0.5,\i-0.5)$) {$\y$};
        \ifnum\i=#3 
            \draw[line width=1.5pt, blue, line cap=round] ($(#1)+(\x,#3)$) -- ($(#1)+(#2,#3)$);
        \fi
        \xdef\te{\x}
    }
}
\newcommand\fillshade[2]{
    \foreach \x/\y in {#2}{
        \path[fill, blue!15] ($(#1)+(\x,\y)$) rectangle +(1,1);
    }
}
\newcommand\markdr[2]{
    \foreach \x/\row in {#2}{
        \node[font=\Large] at ($(#1)+(\x-0.3, \row-0.5)$) {$\star$};
    }
}
\newcommand\touchpoints[2]{
    \foreach \x/\y in {#2}{
        \filldraw ($(#1)+(\x,\y)$) circle (2.5pt);
    }
}

\title{The nonsymmetric compositional Delta  theorem}

\author{Dun Qiu}
\address{Center for Combinatorics, LPMC, Nankai University, Tianjin 300071, P. R. China}
\email{qiudun@nankai.edu.cn}
\thanks{Dun Qiu is supported in part by the Fundamental Research Funds for the Central Universities, the National Natural Science Foundation of China (12271023 and 12171034),
and the Natural Science Foundation of Tianjin (24JCZDJC01390).}

\author{Minhao Zhang}
\address{Center for Combinatorics, LPMC, Nankai University, Tianjin 300071, P. R. China}
\email{2120240009@mail.nankai.edu.cn}
\date{\today}

\begin{document}


\begin{abstract}
Extending the symmetric framework of D'Adderio and Mellit, we establish a nonsymmetric generalization of the compositional Delta theorem. Building on Blasiak et al.'s theory of flagged LLT polynomials, we derive signed and unsigned nonsymmetric identities evaluated in terms of flagged LLT polynomials. Furthermore, by introducing nonsymmetric variants of the $\nabla$ and $\tau^*$ operators, we obtain a novel operator formulation. We show that applying Weyl symmetrization to these nonsymmetric identities systematically recovers the original compositional Delta theorem. Finally, we propose analogous conjectures regarding stable atom positivity.
\end{abstract}

\maketitle
\section{Introduction}
In the 1990s, Garsia and Haiman \cite{GH93} initiated the study of the \(\mathfrak{S}_n\)-module of diagonal harmonics, 
and conjectured that its Frobenius characteristic is exactly \(\nabla e_n\).
Here, \(\nabla\) is the celebrated \emph{nabla operator} acting on the ring of symmetric functions \(\sym[X]\), first introduced in \cite{BG99}.
This conjecture was later proved by Haiman in \cite{Haiman02}.

In 2005, a conjectural combinatorial formula for \(\nabla e_n\) was proposed, 
indexed by labelled Dyck paths (also known as parking functions) by Haglund et al. \cite{HaHaLoReUl05}.
In 2012, the compositional shuffle conjecture was introduced by Haglund, Morse, and Zabrocki \cite{HagMoZa12}, 
which refined the original shuffle conjecture.
Building on these developments, Carlsson and Mellit \cite{cmellit} introduced the celebrated Dyck path algebra in 2018 and 
proved the compositional shuffle conjecture, thereby settling the longstanding shuffle conjecture.
The shuffle theorem admits a far-reaching generalization, known as the $(km,kn)$-compositional shuffle theorem.
Its algebraic framework was introduced by Gorsky and Negu\c{t} \cite{GorskyNegut15}, 
while the full combinatorial statement was formulated by Bergeron et al. in \cite{BeGaSeXi16}. 
This conjecture was subsequently proved by Mellit in \cite{mellit}.
For a more comprehensive overview of the recent progress on the shuffle theorem and its generalizations, 
we refer to 
\cite{BlasiakHaimanMorsePunSeelinger24, BlasiakHaimanMorsePunSeelinger25, BlasiakHaimanMorsePunSeelinger25b, 
	BlasiakHaimanMorsePunSeelinger23b, KimOh26, KimLeeOh25, KimOh24}.

Around the same time, the Delta conjecture was introduced by Haglund, Remmel, and Wilson \cite{HaglundRemmelWilson18}, which admits both a rise version and a valley version.
In 2021, D'Adderio, Iraci, and Vanden Wyngaerd \cite{theta} introduced the \(\Theta\) operator and generalized the Delta conjecture 
to the compositional Delta conjecture, which can be stated as
\[\Theta_{e_k} \nabla C_{\alpha}(X;t)=\mathop{\sum_{P\in \LD(n)^{\ast k}}}_{\dcomp(P)=\alpha}q^{\area(P)}t^{\dinv(P)} \xx^P.\]
In 2022, D'Adderio and Mellit \cite{proofdelta} completed the proof of this conjecture, 
thereby proving the rise version of the Delta conjecture.
The Delta conjecture and its compositional refinement have since become central topics in algebraic combinatorics, see
\cite{BlasiakHaimanMorsePunSeelinger23, DaddIraVan19a, HagRhSh18, HagRhSh19, QW, Rhoades18, Romero17, Zabrocki19, Zabrocki19b}.

The above developments all took place in the symmetric setting; parallel progress has also been made in the nonsymmetric world.
Opdam~\cite{Opdam}, Macdonald~\cite{macdonald}, and Cherednik~\cite{cherednik} pioneered an independent line of research 
on nonsymmetric Macdonald polynomials \(E_{\alpha}(x_1, \ldots, x_n; q, t)\) for \(\alpha \in \mathbb{Z}^n\).
This nonsymmetric framework, which combines tools from representation theory and topology, 
has grown into a rich independent research area.
Furthermore, the study of nonsymmetric Macdonald polynomials has shed new light on the symmetric theory, 
see, e.g., \cite{Knop97, Sahiinterpolation, HalversonRam, Baratta, BechtloffWeising23, HagHaiLo08}.
However, early work only extended the integral form of Macdonald polynomials, \(J_\mu(X;q,t)\), to the nonsymmetric setting, 
but failed to generalize the modified Macdonald polynomials \(\tilde{H}_\mu(X;q,t)\).
Recently, Blasiak et al. \cite{flaggedllt} introduced the plethysm operator \(\Pisf_\ell\), 
which parallels the symmetric operation \(f[X] \mapsto f[X/(1-t)]\).
Using this operator, they constructed the modified \(\ell\)-nonsymmetric Macdonald polynomials \(\tE_{\eta|\lambda}(X; q, t)\), 
solving this longstanding open problem.
In their work, they also introduced flagged LLT polynomials, 
and showed that the modified \(\ell\)-nonsymmetric Macdonald polynomials admit a positive expansion in terms of these flagged LLT polynomials.

Building on the framework of \cite{flaggedllt}, Blasiak et al. \cite{nsshuffle} revisited the algebraic techniques developed in the proofs of the shuffle theorems \cite{cmellit, mellit}.
Specifically, they constructed the \emph{nonsymmetric nabla operator} \(\modnab\) acting on modified \(\ell\)-nonsymmetric Macdonald polynomials, 
which parallels the \(\nabla\) operator in the symmetric setting.
This provided a nonsymmetric, fully parallel counterpart of the $(km,kn)$-compositional shuffle theorem.
Moreover, the combinatorial side of their identity is exactly a positive linear combination 
of the flagged LLT polynomials introduced in \cite{flaggedllt}.

In this paper, building on the framework of \cite{flaggedllt, nsshuffle}, 
we reorganize and rewrite the proof of the compositional Delta conjecture from \cite{theta, proofdelta}.
Using the language of these works, we first establish our first main result, a signed nonsymmetric compositional Delta theorem:

\begin{thm}[Signed nonsymmetric compositional Delta theorem, see Theorem \ref{q ns delta unmod}]
	\begin{equation}\label{eq:signed_delta}
		\alg = \mathop{\sum_{(\pi, dr)\in \mathrm{D}(\alpha)^{\ast k}}} q^{\area(\pi, dr)} \Growpm_\ell(\hat{\pi}', \Sigma_\pi)\,.
	\end{equation}
\end{thm}

By applying the $\ell$-nonsymmetric plethysm operator $\Pisf_\ell$ to this identity, 
we obtain the (unsigned) nonsymmetric compositional Delta theorem, which serves as our second main result:

\begin{thm}[Nonsymmetric compositional Delta theorem, see Theorem \ref{q ns delta mod} and Corollary \ref{c ns delta mod}]
	\begin{equation}\label{eq1}
		\Pisf_\ell \spa \alg
		= \sum_{\pi \in \mathrm{D}(\alpha)^{\ast k}} \sum_{w \in \FWPF'_\alpha(\pi)} q^{\area(\pi, dr)} t^{\dinv'(\pi, w)} \,\xx^{\content(w)}.
	\end{equation}
\end{thm}
We note that the $\dinv'$ statistic here is distinct from the standard $\dinv$ statistic used in the symmetric setting.
Furthermore, by applying Weyl symmetrization to \eqref{eq1}, we successfully recover the $\omega$ statement of the original compositional Delta theorem (see Proposition~\ref{prop:recover_omega}).

Next, we reformulate our nonsymmetric compositional Delta theorem 
in terms of the nonsymmetric nabla operator $\modnab$ from \cite{nsshuffle}.
In the process, we introduce two natural nonsymmetric generalizations of the $\tau_u^*$ operator, denoted $\overline{\tau_{u,\ell}^*}$ and $\underline{\tau_{u,\ell}^*}$. Using these, together with $\modnab$ and the \emph{signed nonsymmetric nabla operator} $\umnab$ (also introduced in \cite{nsshuffle}), we obtain alternative versions of the nonsymmetric compositional Delta theorems:

\begin{thm}[Alternative formulations, see Corollaries \ref{cor:ns_delta_column} and \ref{cor:ns_delta_final}]
	In terms of the signed flagged column LLT polynomials, we have:
	\begin{multline}\label{eq:alt_column}
		(-1)^{n+k} (-t)^{\ell-|\alpha|} \vartheta \left[\overline{\left(\tau_{qtu,\ell}^*\right)}^{-1} \left(\umnab\right)^{-1} \overline{\left(\tau_{qtu,\ell}^*\right)}^{-1} \right]_k \left( x_1^{\alpha_1} \cdots x_\ell^{\alpha_\ell}\right) \\
		= \mathop{\sum_{(\pi, dr)\in \mathrm{D}(\alpha)^{\ast k}}} q^{\area(\pi, dr)} \Gcolpm_\ell(\hat{\pi}', \Sigma_\pi).
	\end{multline}
	Applying plethysm and shifting to the standard formulation yields:
	\begin{equation}\label{eq2}
		(-1)^{n+k} \vartheta \left[\underline{\left(\tau_{qtu,\ell}^*\right)}^{-1} \modnab^{-1} \underline{\left(\tau_{qtu,\ell}^*\right)}^{-1}\right]_k \nsC_\alpha 
		= \sum_{\pi \in \mathrm{D}(\alpha)^{\ast k}} \sum_{w \in \FWPF_\alpha(\pi)} q^{\area(\pi, dr)} t^{\dinv(\pi,w)} \xx^{\content(w)}.
	\end{equation}
\end{thm}

It is worth noting that when $k=0$, our results naturally reduce to the nonsymmetric shuffle theorems. Specifically, setting $k=0$ in \eqref{eq:signed_delta} recovers the $m=n=1$ case of \cite[Theorem 6.3.1]{nsshuffle}. Analogously, specializing \eqref{eq2} to $k=0$ yields the $m=1$ case of \cite[Theorem 1.1.1]{nsshuffle}.

The flagged word parking functions $\FWPF_\alpha(\pi)$ here differ from standard parking functions in that 
they impose an additional flagged condition on the labels of the north steps that touch the main diagonal.
Once again, applying Weyl symmetrization to the unmodified identity allows us to recover the original compositional Delta theorem (see Proposition~\ref{restore}).


Both of our two versions of the nonsymmetric compositional Delta theorem are fully parallel to the classical symmetric case.
A key feature is that the combinatorial sides of both of our identities are given by a positive linear combinations of 
the flagged LLT polynomials from \cite{flaggedllt}, which can also be phrased purely in terms of Dyck path combinatorics.

To help the reader navigate the intricate relationships between the various algebraic spaces (such as $\P(\ell)$, $V_\ell$, and $\sym[X]$) and the diverse operators discussed in this work, we provide a comprehensive summary diagram in Figure~\ref{fig:summary}. We recommend referring to this diagram while following the main algebraic proofs.

The paper is organized as follows. 
In Section~2, we review the necessary background material, including symmetric functions, combinatorial objects, the Dyck path algebra, and the flagged LLT polynomials introduced in \cite{flaggedllt, nsshuffle}. 
In Section~3, using the methods from \cite{proofdelta} combined with the results from Section~\ref{flagged}, we give the proof of our first main identity \eqref{eq1}. We also show how the \(\omega\) statement of the compositional Delta theorem can be recovered by applying Weyl symmetrization to \eqref{eq1}. 
In Section~4, we introduce the nonsymmetric generalizations \(\overline{\tau_{u,\ell}^*}\) and \(\underline{\tau_{u,\ell}^*}\) of the \(\tau_u^*\) operator, together with the nonsymmetric nabla operator \(\modnab\) from \cite{nsshuffle}. We then prove our second main identity \eqref{eq2}, and show that the original compositional Delta theorem can also be recovered by applying Weyl symmetrization to \eqref{eq2}. 
Finally, in Section~5, we define the nonsymmetric \(\Theta\) operator and propose several open conjectures.

\begin{figure}[htbp]\label{fig:summary}
\begin{tikzpicture}[scale=1]
\def\drawLineBelowRow#1#2{
  \draw (#2.west|-#2-#1-1.south) -- (#2.east|-#2-#1-1.south);
}
\def\drawBoundingBox#1#2{
  \draw (#2.west|-#2-#1-1.south) -- (#2.east|-#2-#1-1.south) --
  (#2.east|-#2-1-1.north) -- (#2.west|-#2-1-1.north) -- cycle;
}

   \matrix[matrix of math nodes, nodes={anchor=west}] (nsunmod) at (0,10) {
     \node[align=center, minimum width=6.5cm] (nsunmod-1-1) {$\P(\ell)$};\\
     \text{(1) } \Growpm_\ell(\pi, \Sigma)(X;t) = \chi_\ell^\PP(\pi, \Sigma)(X;t) \\
     \text{(2) } \Gcolpm_\ell(\pi, \Sigma)(X;t^{-1}) = \widetilde{\chi}_\ell^\PP(\pi, \Sigma)(X;t)  \\
     \text{(3) } \stE_{\eta|\lambda}(X;q,t) \\
     \text{(4) } (-t)^{|\alpha|-\ell} x^\alpha\\
     \text{(5) } \alg := (-1)^{n} \PP_\ell^{-1} t^{\ell-|\alpha|} \\ 
     \qquad\qquad\quad\times\M^\Q_k \left( y_1^{\alpha_1-1} \dots y_\ell^{\alpha_\ell-1} d_+^\ell (1)\right)\\
     \text{(6) } (-1)^{n+k} (-t)^{\ell-|\alpha|} \vartheta \\ 
     \left[\overline{\left(\tau_{qtu,\ell}^*\right)}^{-1} \left(\umnab\right)^{-1} \overline{\left(\tau_{qtu,\ell}^*\right)}^{-1} \right]_k \left( x_1^{\alpha_1} \dots x_\ell^{\alpha_\ell}\right)\\
   };
   \drawBoundingBox{9}{nsunmod}
   \drawLineBelowRow{1}{nsunmod}

  \matrix [matrix of math nodes, nodes={anchor=west}] (symunmod) at (0,0) {
    \node[align=center, minimum width=4.5cm] (symunmod-1-1) {$V_\ell$}; \\
    \text{(1) } (-1)^d \chi_\ell(\pi, \Sigma)\\
    \text{(2) } \widetilde{\chi}_\ell(\pi, \Sigma)(X;t)  \\
    \text{(3) } \PP_\ell \big(\stE_{\eta|\lambda}(X;q,t)\big) \\
    \text{(4) } \PP_\ell \big((-t)^{|\alpha|-\ell} x^\alpha\big)\\
    \text{(5) } t^{\ell-|\alpha|} \M^\Q_k \left( y_1^{\alpha_1-1} \dots y_\ell^{\alpha_\ell-1} d_+^\ell (1)\right)\\
    \phantom{\text{(6)}}\\
   };
   \drawBoundingBox{7}{symunmod}
  \drawLineBelowRow{1}{symunmod}

  \matrix [matrix of math nodes, nodes={anchor=west}] (nsmod) at (9.5,10) {
    \node[align=center, minimum width=4.5cm] (nsmod-1-1) {$\P(\ell)$};\\
    \text{(1) } \Grow_\ell(\pi, \Sigma)(X;t) \\
    \text{(2) } \Gcol_\ell(\pi, \Sigma)(X;t^{-1}) \\
    \text{(3) } \tE_{\eta|\lambda}(X;q,t) \\
    \text{(4) } \nsC_\alpha(X;t) \\
    \text{(5) } \Pisf_\ell \spa \alg = (-1)^{n} \Pisf_\ell \PP_\ell^{-1} t^{\ell-|\alpha|} \\ 
    \qquad\qquad\qquad\times\M^\Q_k \left( y_1^{\alpha_1-1} \dots y_\ell^{\alpha_\ell-1} d_+^\ell (1)\right)\\
    \text{(6) } (-1)^{n+k} \vartheta 
    \left[\underline{\left(\tau_{qtu,\ell}^*\right)}^{-1} \modnab^{-1} \underline{\left(\tau_{qtu,\ell}^*\right)}^{-1}\right]_k \nsC_\alpha\\
  };
   \drawBoundingBox{8}{nsmod}
  \drawLineBelowRow{1}{nsmod}

  \matrix [matrix of math nodes, nodes={anchor=west}] (symmod) at (9.5,0) {
    \node[align=center, minimum width=6.5cm] (symmod-1-1) {$\P(0) \cong \sym[X]$}; \\
    \text{(1) } \Grow_0(\mathsf{N}^\ell \pi, \Sigma)(X;t) = \omega \chi(\pi, \Sigma_\pi)(X;t) \\
    \text{(2) } \Gcol_0(\mathsf{N}^\ell \pi, \Sigma)(X;t^{-1}) = \chi(\pi, \Sigma_\pi)(X;t^{-1}) \\
    \text{(3) } \omega \Htild_\mu(X;q,t) \\
    \text{(4) } C_\alpha(X;t^{-1})\\
    \text{(5) } \omega \Theta_{e_k} \nabla C_\alpha(X;t)\\
    \text{(6) } \Theta_{e_k} \nabla C_\alpha(X;t)\\
  };
  \drawBoundingBox{7}{symmod}
  \drawLineBelowRow{1}{symmod}

  \draw[->, shorten >=0.1cm, shorten <=0.1cm] (nsunmod) -- (symunmod) node[midway, left] {$\PP_\ell$} node[midway, right] {$\cong$};
  \draw[->, shorten >=0.1cm, shorten <=0.1cm] (nsunmod) -- (nsmod) node[midway, above] {$\Pisf_\ell$};
  \draw[->, shorten >=0.1cm, shorten <=0.1cm] (nsmod) -- (symmod) node[midway, right] {$\Weyl_1 \dots \Weyl_\ell$};
  \draw[->, shorten >=0.2cm, shorten <=0.1cm] (symunmod) -- (symmod) node[midway, above] {$(-1)^d$} node[midway, below] {$\omega d_-^\ell$};

\end{tikzpicture}
\caption{This diagram summarizes the core relations among the main concepts and theorems established in this paper. Restricting to any specific row yields a commutative diagram among the corresponding objects. In particular, row (5) encapsulates the main results of Section~\ref{secnscdt}, while row (6) represents the core content of Section~\ref{secnnnf}.}
\end{figure}
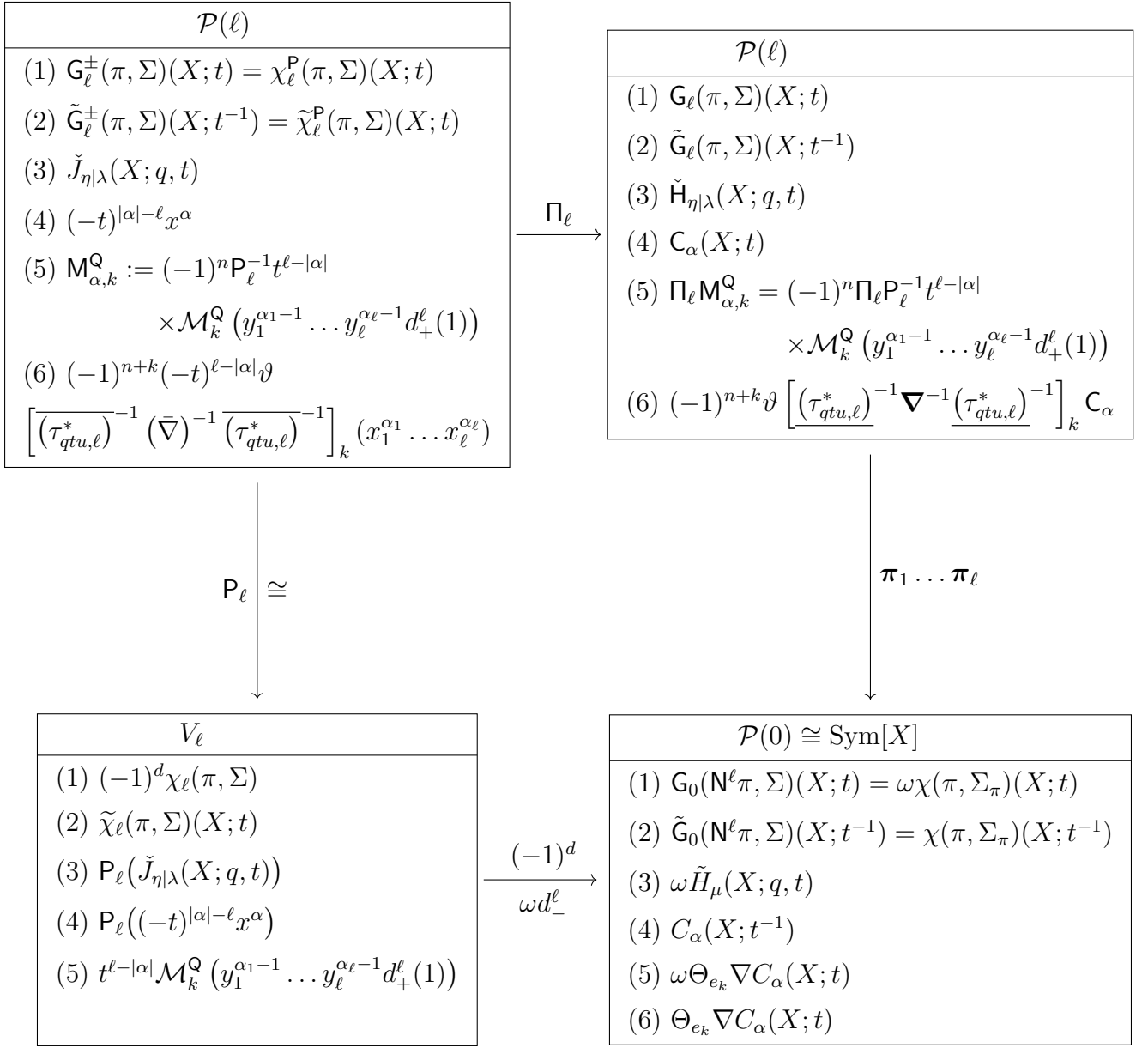

\section{Preliminaries}
In this section, we review the necessary background and fix the notation used throughout this work. We briefly recall the basics of symmetric functions (Section~\ref{subsec:sym}), introduce the combinatorial objects related to the compositional Delta theorem (Section~\ref{combinatorialobj}), and review the algebraic framework of the Dyck path algebra (Section~\ref{subsec:algebra}). We then discuss nonsymmetric plethysm and Weyl symmetrization (Section~\ref{nonmac}), and finally provide a brief overview of flagged LLT polynomials (Section~\ref{flagged}).

\subsection{Symmetric functions}
\label{subsec:sym}
For the foundational theory of symmetric functions, we follow the classical text by Macdonald \cite{mac95} and the lectures of Haglund \cite{hag08}.

A \emph{partition} $\lambda = (\lambda_1, \dots, \lambda_\ell)$ is a finite weakly decreasing sequence of non-negative integers, while a \emph{composition} $\alpha = (\alpha_1, \dots, \alpha_\ell)$ is an ordered sequence of positive integers. Let \(\Par\) denote the set of all partitions.
Let $\gamma = (\gamma_1, \dots, \gamma_\ell)$ denote either a partition or a composition. The elements $\gamma_i$ are called \emph{parts}. 
Its \emph{size} is defined as $|\gamma| := \sum_{i=1}^\ell \gamma_i$, and its \emph{length} $\ell(\gamma)$ is the number of positive parts. 
We write $\lambda \vdash n$ to indicate that $\lambda$ is a partition of size $n$, and $\alpha \vDash n$ to denote that $\alpha$ is a composition of size $n$.

We use French notation to represent the \emph{Young diagram} of a partition \(\lambda\), where each row corresponds to a part \(\lambda_j\), left-justified and ordered from bottom to top. 
For a cell \(x = (i,j)\) in the \(i\)-th column and \(j\)-th row, we define its \emph{arm} \(a_\lambda(x)\), \emph{leg} \(l_\lambda(x)\), \emph{coarm} \(a'_\lambda(x)\), and \emph{coleg} \(l'_\lambda(x)\) as follows: 
\(a_\lambda(x)\) and \(a'_\lambda(x)\) denote the number of cells in the same row strictly to the east and west of \(x\), respectively; 
similarly, \(l_\lambda(x)\) and \(l'_\lambda(x)\) denote the number of cells in the same column strictly to the north and south of \(x\), respectively (see Figure~\ref{fig:arm}).

\begin{figure}[!ht]
    \centering
    \begin{tikzpicture}[scale = 0.5]
        \fill[green] (0,3) rectangle (3,4); 
        \fill[green] (4,3) rectangle (8,4); 
        \fill[green] (3,0) rectangle (4,3); 
        \fill[green] (3,4) rectangle (4,6); 
        \FrenchYD{9, 9, 8, 8, 6, 5, 3}
        \node at (1.5, 3.5) {$a'$};
        \node at (5.5, 3.5) {$a$};
        \node at (3.5, 1.5) {$l'$};
        \node at (3.5, 4.5) {$l$};
        \node at (3.5, 3.5) {$x$}; 
    \end{tikzpicture}
    \caption{Arm, leg, coarm, and coleg of a cell $x$ in a Young diagram.}
    \label{fig:arm}
\end{figure}
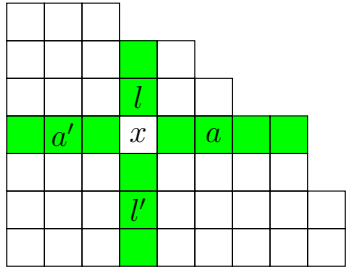

Let \(\Bbbk=\mathbb{Q}(q,t)\) be the coefficient field.  
Denote by \(\sym[X]=\Lambda_{\Bbbk}(X)\) the algebra of symmetric functions in an infinite alphabet \(X=x_1,x_2,\dots\) with coefficients in \(\Bbbk\). 
We recall several standard graded bases of \(\sym[X]\), indexed by partitions \(\lambda\): the elementary symmetric functions \(e_\lambda\), the complete homogeneous symmetric functions \(h_\lambda\), the power-sum symmetric functions \(p_\lambda\), and the Schur functions \(s_\lambda\). 

The \emph{Hall inner product} \(\langle \cdot , \cdot \rangle\) on \(\sym[X]\) can be defined by \(\langle s_\lambda , s_\mu \rangle = \delta_{\lambda\mu}\). 
Let \(\omega : \sym[X] \to \sym[X]\) be the involutory \(\Bbbk\)-algebra automorphism defined by \(\omega(s_\lambda) = s_{\lambda^*}\).

For a partition \(\mu \vdash n\), let \(\Htild_\mu[X; q, t]\) denote the \emph{modified Macdonald polynomials}. 
When the variables and parameters are clear from the context, we simply write \(\Htild_\mu\) or \(\Htild_\mu[X]\). 
The modified Macdonald polynomials can be expanded in terms of Schur functions as
\begin{equation}
\Htild_\mu[X; q, t] = \sum_{\lambda \vdash n} \widetilde{K}_{\lambda\mu}(q, t) s_\lambda,
\end{equation}
where the \emph{modified Kostka coefficients} $\widetilde{K}_{\lambda\mu}(q, t)$ are related to the standard Kostka--Macdonald coefficients $K_{\lambda\mu}(q, t)$ by 
\begin{equation}
\widetilde{K}_{\lambda\mu}(q, t) = K_{\lambda\mu}(q, 1/t) t^{n(\mu)}, \quad \text{with} \quad n(\mu) = \sum_{i \ge 1} \mu_i(i-1).
\end{equation}
The set \(\{\Htild_\mu\}_\mu\) forms a basis of \(\sym[X]\) with coefficients in \(\mathbb{Q}(q,t)\). 
This basis is a modification of the one introduced by Macdonald \cite{mac95}. 
Haiman \cite{Haiman02} proved that these polynomials are the Frobenius characteristics of the Garsia--Haiman modules \cite{GH93}. 

Plethystic notation provides a convenient way to denote substitutions in symmetric function identities. 
Recall that the power-sum symmetric functions $\{p_k\}_{k \ge 1}$ generate $\sym[X]$ as a $\Bbbk$-algebra. 
For any expression $A$ involving indeterminates (including the ground field parameters $q,t$), we define $p_k[A]$ as the result of replacing every indeterminate $a$ in $A$ with $a^k$. 
For any symmetric function $f$, expressed as a polynomial in the $p_k$'s, we define $f[A]$ by substituting $p_k[A]$ for each $p_k$. 
The map $f \mapsto f[A]$ thus defines a ring homomorphism from $\sym[X]$ to the ring of expressions in the indeterminates of $A$.

For the infinite alphabet $X = \{x_1, x_2, \dots\}$, we adopt the standard plethystic convention where the symbol $X$ inside the brackets represents the formal sum $x_1 + x_2 + \cdots$. 
Under this convention, we have $f[X] = f(x_1, x_2, \dots)$. 
A key role is played by the {generating function for homogeneous symmetric functions}
\begin{equation}
\Omega[X] := \sum_{n \ge 0} h_n[X] = \exp\left(\sum_{k \ge 1} \frac{p_k[X]}{k}\right).
\end{equation}
Using plethystic notation, the classical \emph{Cauchy identity} can be expressed in the following compact form:
\begin{equation}
\Omega[XY] = \sum_{\lambda} s_\lambda[X] s_\lambda[Y],
\end{equation}
where $XY$ denotes the product of the two alphabets $X$ and $Y$.

Let $M := (1-q)(1-t)$. For a partition $\mu$, we define the following statistics:
\begin{equation*}
\begin{aligned}
    B_{\mu} &:= \sum_{c\in \mu} q^{a_{\mu}'(c)} t^{l_{\mu}'(c)}, \quad T_{\mu} := \prod_{c\in \mu} q^{a_{\mu}'(c)} t^{l_{\mu}'(c)}, \\
    \Pi_{\mu} &:= \prod_{c\in \mu \setminus \{(1,1)\}} (1-q^{a_{\mu}'(c)} t^{l_{\mu}'(c)}),
\end{aligned}
\end{equation*}
where the empty product is taken to be $1$. For any symmetric function $f[X]$, we also denote 
\begin{equation*}
f^*[X] := f[X/M].
\end{equation*}

The \emph{nabla operator} $\nabla$, introduced in \cite{BG99}, is defined by its action on the modified Macdonald basis as
\begin{equation*}
    \nabla \Htild_{\mu} = T_{\mu} \Htild_{\mu} \quad \text{for all } \mu. 
\end{equation*}
We also employ a sign-twisted variant $\nabla'$, defined by $\nabla' \Htild_{\mu} = (-1)^{|\mu|} T_{\mu} \Htild_{\mu}$. 
This adjusted operator played a central role in the proofs of the compositional shuffle theorem \cite{cmellit} and the compositional Delta theorem \cite{proofdelta}.

The \emph{Delta operators} $\Delta_f$ and $\Delta'_f$ were introduced in \cite{BGHT99}. They act on the modified Macdonald basis as 
\begin{equation*}
    \Delta_f \Htild_{\mu} = f[B_{\mu}] \Htild_{\mu}, \quad \Delta_f' \Htild_{\mu} = f[B_{\mu}-1] \Htild_{\mu} \quad \text{for all } \mu.
\end{equation*}
Let \(\sym[X]^{(n)}\) denote the space of homogeneous symmetric functions of degree \(n\). 
On this space, the nabla operator and Delta operators are related by \[\nabla = \Delta_{e_n} = \Delta'_{e_{n-1}}.\]
Furthermore, for \(1\leq k\leq n\), these operators satisfy the identity
\begin{equation} \label{eq:deltaprime}
\Delta_{e_k} = \Delta_{e_k}' + \Delta_{e_{k-1}}',
\end{equation}
where by convention \(\Delta_{e_k} = \Delta_{e_{k-1}}' = 0\) for \(k > n\).

Following \cite{theta}, for any symmetric function $f$, we define the \emph{Theta operator} $\Theta_f$ on $\sym[X]$ by
\[
    \Theta_f F[X] := \mathbf{\Pi} f^* \mathbf{\Pi}^{-1} F[X],
\]
where the operator $\mathbf{\Pi}$ acts on the modified Macdonald basis by $\mathbf{\Pi} \Htild_\mu := \Pi_\mu \Htild_\mu$.

For \(a \geq 0\), we define the \emph{\(C\) operator} \(C_a\) acting on \(\sym[X]\) by
\begin{equation*}
    C_a f[X] := \left< z^a \right> \left(- \cfrac{1}{t}\right)^{a-1} f\left[X - \cfrac{1-1/t}{z}\right] \Omega [zX].
\end{equation*}
For a composition \(\alpha = (\alpha_1, \ldots, \alpha_\ell)\), the {symmetric function} \(C_\alpha(X;t)\) is defined by the successive action of these operators on the constant \(1\):
\begin{equation*}
    C_\alpha(X;t) := C_{\alpha_1} C_{\alpha_2} \ldots C_{\alpha_\ell}(1).
\end{equation*}

\subsection{Combinatorial objects}
\label{combinatorialobj}
For the combinatorial objects, we adopt the notation established by D'Adderio and Mellit in their work on the compositional Delta theorem \cite{theta, proofdelta}.

Let $n$ be a positive integer. An $(n,n)$-Dyck path $\pi$ is a lattice path consisting of north and east steps from $(0,0)$ to $(n,n)$ that stays weakly above the main diagonal $y=x$.
We denote the set of all such $(n,n)$-Dyck paths by $\mathrm{D}(n)$.

A \emph{labelled Dyck path} (classically known as a \emph{parking function}) is a Dyck path $\pi \in \mathrm{D}(n)$ whose north steps are labeled with positive integers such that the labels strictly increase along each column from bottom to top.
We denote the set of all such labelled paths by $\LD(n)$. 
Throughout this paper, for any labelled Dyck path $P \in \LD(n)$, we consistently use $\pi$ to denote its underlying Dyck path.

For an $(n,n)$-Dyck path $\pi$, we define its \emph{area word} to be the length-$n$ sequence of non-negative integers $a(\pi) = (a_1(\pi), \dots, a_n(\pi))$,
where $a_i(\pi)$ counts the number of full cells between the path and the main diagonal in the $i$-th row (read from bottom to top).

The \emph{double rises} of $\pi$, which geometrically correspond to the indices of consecutive north (NN) steps, are defined as
\begin{equation*}
    r(\pi) := \{i : a_i(\pi) = a_{i-1}(\pi) + 1\}.
\end{equation*}

A \emph{decorated Dyck path} is a pair $(\pi, dr)$, where $\pi \in \mathrm{D}(n)$ and $dr \subseteq r(\pi)$ is a chosen subset of its double rises. 
We denote the set of all decorated Dyck paths with exactly $k$ decorated double rises by $\mathrm{D}(n)^{\ast k}$:
\begin{equation*}
    \mathrm{D}(n)^{\ast k} := \{ (\pi, dr) : \pi \in \mathrm{D}(n),\ dr \subseteq r(\pi),\ |dr|=k \}.
\end{equation*}

Analogously, a \emph{labelled decorated Dyck path} is a pair $(P, dr)$, where $P \in \LD(n)$ and $dr \subseteq r(\pi)$. 
We denote the set of all such paths with $k$ decorated double rises by $\LD(n)^{\ast k}$:
\begin{equation*}
    \LD(n)^{\ast k} := \{ (P, dr) : P \in \LD(n),\ dr \subseteq r(\pi),\ |dr|=k \}.
\end{equation*}

For a decorated Dyck path $(\pi, dr) \in \mathrm{D}(n)^{\ast k}$, its \emph{diagonal composition} $\dcomp(\pi, dr)$ is defined to be a composition of $n-k$. 
Following \cite{proofdelta}, its $i$-th part counts the number of rows between the $i$-th and $(i+1)$-th north steps of $\pi$ that lie on the main diagonal $y=x$, which have not been decorated by elements in $dr$.

For a labelled decorated Dyck path $(P, dr) \in \LD(n)^{\ast k}$, its diagonal composition is simply defined as that of its underlying decorated path, namely $\dcomp(P, dr) := \dcomp(\pi, dr)$.

Given a composition $\alpha\vDash n-k$, we further define the subsets of paths with diagonal composition $\alpha$ as:
\begin{equation*}
\begin{aligned}
    \mathrm{D}(\alpha)^{\ast k} &:= \{ (\pi, dr) \in \mathrm{D}(n)^{\ast k} : \dcomp(\pi, dr) = \alpha \}, \\
    \LD(\alpha)^{\ast k} &:= \{ (P, dr) \in \LD(n)^{\ast k} : \dcomp(P, dr) = \alpha \}.
\end{aligned}
\end{equation*}

Geometrically, the area of a decorated Dyck path is the total number of full cells between the path and the main diagonal, excluding those in the decorated rows. Formally, for a decorated Dyck path $(\pi, dr) \in \mathrm{D}(n)^{\ast k}$, we define
\begin{equation*}
    \area(\pi, dr) := \sum_{i \notin dr} a_i(\pi).
\end{equation*}
For a labelled decorated Dyck path $(P, dr) \in \LD(n)^{\ast k}$, its area is simply defined to be the area of its underlying decorated Dyck path, i.e., $\area(P, dr) := \area(\pi, dr)$.

Finally, let us define the diagonal inversion ($\dinv$) statistic. For a labelled decorated Dyck path $(P, dr) \in \LD(n)^{\ast k}$, this statistic depends only on its underlying parking function $P$. Let $w_i$ denote the label of the north step in the $i$-th row. 
We say that a row index $i$ \emph{attacks} a row index $j$ (where $i < j$) if they lie on the same diagonal relative to the main diagonal, or if row $i$ is exactly one diagonal further from the main diagonal than row $j$. That is, 
\begin{equation*}
    a_i(\pi) = a_j(\pi) \quad \text{or} \quad a_i(\pi) = a_j(\pi) + 1.
\end{equation*}

The $\dinv$ statistic of $(P, dr)$ is defined as the number of attacking pairs $(i,j)$ with $i<j$ that create an inversion based on their labels. Specifically:
\begin{equation*}
\begin{aligned}
    \dinv(P, dr) := &\# \{ 1\leq i< j\leq n : a_i(\pi) = a_j(\pi) \text{ and } w_i < w_j \} \\
    &+ \# \{ 1\leq i< j\leq n : a_i(\pi) = a_j(\pi)+1 \text{ and } w_i > w_j \}.
\end{aligned}
\end{equation*}

Figure~\ref{fig:statistics_example} provides a comprehensive illustration of a non-trivial size 8 labelled decorated Dyck path $(\pi, dr)$ with $k=2$ decorated double rises, detailing the calculation of all defined statistics. In this example, $\dcomp(P, dr) = (3, 3)$. The attack relationships are also highlighted, contributing to a $\dinv$ of 8.

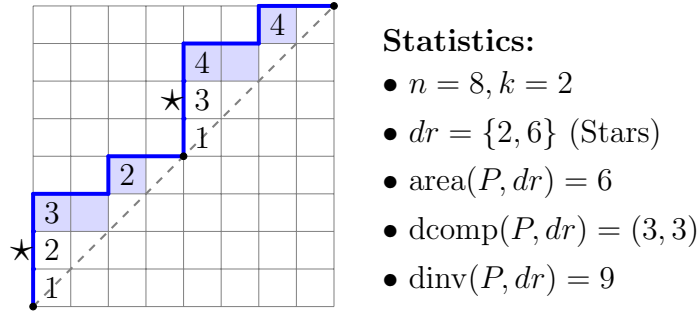
\begin{figure}[!ht]
    \centering
    \begin{tikzpicture}[scale=0.5]
        \fillshade{0,0}{0/2, 1/2, 2/3, 4/6, 5/6, 6/7}
        \PFmn{0,0}{8}{8}{0/1, 0/2, 0/3, 2/2, 4/1, 4/3, 4/4, 6/4}
        \markdr{0,0}{0/2, 4/6}
        \touchpoints{0,0}{0/0, 4/4, 8/8}
        \node[anchor=west] at (9, 7.1) {\textbf{Statistics:}};
        \node[anchor=west] at (9, 5.9) {$\bullet$ $n = 8, k = 2$};
        \node[anchor=west] at (9, 4.6) {$\bullet$ $dr = \{2, 6\}$ (Stars)};
        \node[anchor=west] at (9, 3.3) {$\bullet$ $\area(P, dr) = 6$};
        \node[anchor=west] at (9, 2.0) {$\bullet$ $\dcomp(P, dr) = (3, 3)$};
        \node[anchor=west] at (9, 0.7) {$\bullet$ $\dinv(P, dr) = 9$};
    \end{tikzpicture}
    \caption{Example of a labelled decorated Dyck path $(P, dr) \in \LD(3, 3)^{*2}$.}
    \label{fig:statistics_example}
\end{figure}

Figure~\ref{fig:statistics_example} provides an example of a size 8 labelled decorated Dyck path $(P, dr) \in \LD(3, 3)^{*2}$ with $2$ decorated double rises. In this example, the diagonal composition is $\dcomp(P, dr) = (3, 3)$. The area cells contributing to $\area(P, dr)$ are shaded. There are $9$ attacking pairs $(i, j)$ with $i<j$ that contribute to the $\dinv$ statistic: $\{(2,5), (2,6), (2,8), (3,4), (3,7), (4,5), (4,6), (4,8), (6,8)\}$.

\subsection{The Dyck path algebra}
\label{subsec:algebra}
The Dyck path algebra $\AA_{t,q}$, first introduced and studied by Carlsson and Mellit \cite{cmellit, mellit}, admits a stable limit version developed by Ion and Wu \cite{IonWu}. 
In this subsection, we follow the notation of \cite{nsshuffle}, working primarily with the conventions established by Mellit \cite{mellit}. 
However, since the original conventions from Carlsson and Mellit \cite{cmellit} are also relevant, we add a superscript \(CM\) to distinguish their operators from ours.

Following the conventions in \cite{nsshuffle}, we introduce the operators of the \emph{Dyck path algebra} \(\AA = \AA_t\). Note that, compared to the original construction in \cite{mellit}, the roles of the parameters $q$ and $t$ are interchanged to maintain consistency with the nonsymmetric Macdonald theory.

\begin{defn}
We define the vector space $V_\ell := \mathbb{Q}(q,t)[y_1, \dots, y_\ell] \otimes \sym[X]_{\QQ(q,t)}$ for each $\ell \in \mathbb{N}$. We then define the operators $T_i \colon V_\ell \to V_\ell$ for $1 \le i \le \ell-1$, $d_+ \colon V_\ell \to V_{\ell+1}$, and $d_- \colon V_\ell \to V_{\ell-1}$ as follows:
\begin{equation}\label{defn1}
    \begin{aligned}
    T_i F & := s_i F + (1-t) y_i \frac{F - s_i F}{y_i - y_{i+1}}, \\[1ex]
    d_- F & := \langle y_\ell^0 \rangle F[y_1, \dots, y_\ell; X - (t-1)y_\ell] \Omega[-y_\ell^{-1} X], \\[1ex]
    d_+ F & := -T_1 \dots T_\ell \big( y_{\ell+1} F[y_1, \dots, y_\ell; X + (t-1)y_{\ell+1}] \big),
    \end{aligned}
\end{equation}
where $F[y_1, \dots, y_\ell; X] \in V_\ell$. The operators $y_i$ are defined on $V_\ell$ recursively by:
\begin{equation}
    y_1 := \frac{1}{t^{\ell-1}(t-1)} (d_+ d_- - d_- d_+) T_{\ell-1} \dots T_1 \quad \text{and} \quad y_{i+1} := t \, T_i^{-1} y_i T_i^{-1}, \quad 1 \le i \le \ell-1.
\end{equation}
\end{defn}

Additionally, the explicit formulas in \eqref{defn1} imply that the action of the operator $y_i$ on $V_\ell$ corresponds simply to multiplication by the polynomial variable $y_i$.

Let \(V_* := \bigoplus_{\ell \ge 0} V_\ell\) be the direct sum of these spaces. We write \(\AA^* = \AA_{t^{-1}}\) for the algebra generated by the elements $d^*_-$, $d_+^*$, and $T_i^*$. As described by Mellit in \cite[Proposition 3.4]{mellit}, $\AA^*$ admits an action on $V_*$ given by:
\begin{equation}\label{defn2}
    d^*_- F := d_- F, \quad T^*_i F := T_i^{-1} F, \quad d_+^* F := \gamma \big( F[y_1, \dots, y_\ell; X + (t-1)y_{\ell+1}] \big),
\end{equation}
where \(\gamma \colon V_{\ell+1} \to V_{\ell+1}\) is the shift operator defined by 
\[\gamma \big( G[y_1, \dots, y_{\ell+1}; X] \big) := G[y_2, \dots, y_{\ell+1}, qy_1; X].\]

There is a natural isomorphism from $\AA$ to $\AA^*$ mapping each generator to its starred counterpart and sending \(t \mapsto t^{-1}\). We denote by $z_i$ the image of $y_i$ under this isomorphism. Explicitly, the operators $z_i$ acting on $V_\ell$ are given by:
\[
    z_1 := \frac{t^\ell}{1-t} (d^*_+ d_- - d_- d^*_+) T^{-1}_{\ell-1} \dots T^{-1}_1 \quad \text{and} \quad z_{i+1} := t^{-1} \, T_i z_i T_i \quad \text{for } 1 \le i \le \ell-1.
\]

The full Dyck path algebra $\AA_{t,q}$ is defined as the quotient of the free product of $\AA$ and $\AA^*$ by the relations given in \cite[Definition 7.1]{cmellit}.

One of the relations of $\AA_{t,q}$ that we shall use later, considered as an operator equation acting on $V_\ell$, is:
\begin{equation}\label{relation}
    z_1 d_+ = -q t^{\ell+1} y_1 d_+^*.
\end{equation}

\begin{prop}[{\cite[Proposition 3.2.1]{nsshuffle}}]
The operators defined in \eqref{defn1} and \eqref{defn2} give a quiver representation of $\AA_{t,q}$ on the space $V_*$.
\end{prop}


The CM operators are related to our operators via the correspondence described in \cite[Remark 3.11]{mellit}.
This correspondence is given by the $\mathbb{Q}(q,t)$-linear isomorphism \( \Q_\ell \colon y_1\cdots y_\ell V_\ell \to V_\ell \), 
defined by \( \Q_\ell(y_1\cdots y_\ell v) := v \).

\begin{prop}[{\cite[Proposition 3.2.2]{nsshuffle}}]
The space $(y_1\cdots y_\ell V_\ell)_{\ell\geq0}$ forms a submodule of the action of $\AA_{t,q}$ on $V_*$. 
Moreover, the following operator identities hold on $V_*$:
\begin{equation}\label{CMM}
    \begin{aligned}
    T_i^{CM} &= \Q_\ell T_i \Q_\ell^{-1}, \\
    d_-^{CM} &= \Q_{\ell-1} (-d_-) \Q_\ell^{-1}, \\
    d_+^{CM} &= \Q_{\ell+1} (-d_+) \Q_\ell^{-1}, \\
    y^{CM}_i &= \Q_\ell y_i \Q_\ell^{-1},\\
    d_+^{* CM} &= \Q_{\ell+1} (y_1 d_+^*) \Q_\ell^{-1}.
    \end{aligned}
\end{equation}
\end{prop}

Let \(\P(\ell) := \QQ(q,t)[x_1,\dots,x_\ell] \otimes \sym[X]_{\QQ(q,t)}(x_{\ell+1}, \dots)\). 
By definition, an element \(f[X] \in \P(\ell)\) means that $f$ is allowed to be nonsymmetric 
in the first $\ell$ variables $x_1, \dots, x_\ell$, 
but symmetric in the remaining variables $x_{\ell+1}, x_{\ell+2}, \dots$.
We use the notation $f[X_N]$ to denote the truncation \(f[x_1, \dots, x_N, 0, 0, \dots]\).
We work with the action on $\P(\ell)$ given in \cite[Eq.~(41)]{nsshuffle}, which is induced by Mellit's action.

We first recall the isomorphism $\PP_\ell$ on $\P(\ell)$, defined in \cite[Eq.~(40)]{nsshuffle}, by:
\begin{equation}\label{isopr}
    \begin{aligned}
        \PP_\ell \colon \P(\ell) &\xrightarrow{\ \cong \ } V_\ell, \\
        x_1^{\eta_1}\cdots x_\ell^{\eta_\ell} g[x_{\ell+1}+ \cdots ] &\mapsto y_1^{\eta_1}\cdots y_\ell^{\eta_\ell} g[X/(t-1)],
    \end{aligned}
\end{equation}
where $g$ is a symmetric function, and \(\eta_1,\dots, \eta_\ell \in \mathbb{N}\).

Via the isomorphism \eqref{isopr}, the induced action on $\P(\ell)$ is then given by:
\[
    d_-^\PP := \PP_{\ell-1}^{-1} \, d_- \, \PP_\ell,  \quad \text{and} \quad d_+^\PP := \PP_{\ell+1}^{-1} \, d_+ \, \PP_\ell.
\]

Denote by \(\overline{\omega}\) the automorphism of \(\sym[X]\) defined by mapping \(F[X;q,t] \mapsto F[-X;q^{-1},t^{-1}]\).
Carlsson and Mellit \cite{cmellit} extended the operator \(\nabla' \overline{\omega}\) from \(\sym[X]\) to \(V_*\).

\begin{thm}[{\cite[Theorem 7.4]{cmellit}}]
    There exists a unique involution \(\N\) on \(V_*\) satisfying
    \[
        \N(1) = 1, \quad \N d^{CM}_+ = d^{* CM}_+ \N, \quad \N d^{CM}_- = d^{CM}_- \N,\quad \N T^{CM}_i = (T_i^{CM})^{-1} \N.
    \]
    In particular, when restricted to \(V_0 = \sym[X]\), this operator coincides with \(\nabla' \overline{\omega}\).
\end{thm}

In their proof of the compositional Delta conjecture, D'Adderio and Mellit \cite{proofdelta} introduced the following useful notation from the work of Garsia and Mellit \cite{fiveterm}.

Let \(u\) be a formal variable. 
\[
    \Theta(u) := \sum_{n=0}^\infty (-u)^n \Theta_{e_n},\quad \Delta(u) := \sum_{n=0}^\infty (-u)^n \Delta_{e_n}, \quad \text{and}\quad 
    \tau^*_u := \sum_{n=0}^\infty (-u)^n \underline{e}_n^*,
\]
where the operator \(\underline{e}_n^*\) is simply defined by \(\e_n^*(f) := e_n^*f\) for every symmetric function \(f\).

Using the classical identity \(\sum_{i=0}^n (-1)^i e_{i}h_{n-i} = \delta_{n,0}\), one readily obtains
\[
    \left(\tau_u^*\right)^{-1} = \sum_{n=0}^{\infty} u^n \underline{h}_n^*.
\]

Similarly, in the proof of \cite[Proposition 3.13]{mellit}, Mellit extended the operator \(\tau_u^*\) to \(V_\ell\), defined by:
\begin{equation}\label{deftauell}
    \tau_{u,\ell}^* := \sum_{n=0}^\infty (-u)^n \e_n\left[\frac{X+(t-1)\sum_{i=1}^\ell y_i}{M}\right].
\end{equation}

\begin{prop}[{\cite[Proposition 3.24]{mellit}}]\label{propnewtau}
With the above definition, \(\tau_{u,\ell}^*\) satisfies the following intertwining relations acting on $V_\ell$:
\begin{equation}\label{eqnewtau}
    \begin{aligned}
        \tau_{u,\ell-1}^* d_- & = d_- \tau_{u,\ell}^* & \tau_{u,\ell+1}^* d_+ & = d_+ \tau_{u,\ell}^* \\[1ex]
        \tau_{u,\ell}^* T_i & = T_i \tau_{u,\ell}^* & \tau_{u,\ell}^* y_i & = y_i \tau_{u,\ell}^* \\[1ex]
        d^{*}_+ \tau^*_{u,\ell} & = (1-u y_1) \tau^*_{u,\ell+1} d^{*}_+ & z_1 \tau^*_{u,\ell} & = (1-uy_1) \tau^*_{u,\ell} z_1.
    \end{aligned}
\end{equation}
\end{prop}

We also note the corresponding intertwining relation for the inverse operator:
\begin{equation}\label{eqtau-1}
    d_- \left(\tau_{u,\ell}^*\right)^{-1} = \left(\tau_{u,\ell-1}^*\right)^{-1} d_-.
\end{equation}
Indeed, this follows from a direct computation utilizing the relation $\tau_{u,\ell-1}^* d_- = d_- \tau_{u,\ell}^*$:
\begin{align*}
    d_- \left(\tau_{u,\ell}^*\right)^{-1} 
    &= \left(\tau_{u,\ell-1}^*\right)^{-1} \tau_{u,\ell-1}^* d_- \left(\tau_{u,\ell}^*\right)^{-1} \\[1ex]
    &= \left(\tau_{u,\ell-1}^*\right)^{-1} d_- \tau_{u,\ell}^* \left(\tau_{u,\ell}^*\right)^{-1} \\[1ex]
    &= \left(\tau_{u,\ell-1}^*\right)^{-1} d_-.
\end{align*}

To prove the compositional Delta conjecture, D'Adderio and Mellit \cite{proofdelta} further extended 
the operator \(\Theta(u) \nabla' \overline{\omega}\) to \(V_\ell\), defined in terms of the operator \(\N\) introduced above.

\begin{prop}[{\cite[Proposition 6.4]{proofdelta}}]\label{originalM}
Let \(\M\) be the antilinear operator on \(V_\ell\) which, when restricted to \(V_0 = \sym[X]\), coincides with \(\Theta(u) \nabla' \overline{\omega}\).
    \begin{equation}\label{defnM}
        \M := \tau^*_{u,\ell} \N \left(\tau^*_{qtu,\ell}\right)^{-1}.
    \end{equation}
It satisfies the following intertwining relations:
    \[
    \begin{aligned}
        \M d^{CM}_- & = d^{CM}_- \M & \M T^{CM}_i & = (T_i^{CM})^{-1} \M \\[1ex]
        \M d^{CM}_+ & = (1-uy_1^{CM})^{-1} d_+^{* CM} \M & \M y_1^{CM} & = (1-uy_1^{CM})^{-1} z_1^{CM} \M.
    \end{aligned}
    \]
\end{prop}

Building on the relations established above, we obtain the following intertwining properties for the operator $\M^\Q := \Q_\ell^{-1} \M \Q_\ell$.

\begin{prop}\label{prenewM}
    We have the following identities acting on the space \((y_1 \dots y_\ell V_\ell)_{\ell \ge 0}\):
    \begin{equation}\label{eqprenewM}
        \M^\Q T_i  = T^{-1}_i \M^\Q, \quad \M^\Q d_- = d_- \M^\Q, \quad  \M^\Q d_+ = \frac{-y_1}{1-uy_1} d^{*}_+ \M^\Q.
    \end{equation}
\end{prop}

\begin{proof}
    The proof follows from direct computations. 

    For $T_i$, we have:
    \[
        \M^\Q T_i = \Q_\ell^{-1} \M \Q_\ell T_i = \Q_\ell^{-1} \M T_i^{CM} \Q_\ell = \Q_\ell^{-1} (T_i^{CM})^{-1} \M \Q_\ell = T_i^{-1} \M^\Q.
    \]

    For $d_-$, utilizing \eqref{CMM}, we deduce:
    \begin{align*}
        \M^\Q d_- &= \Q_{\ell-1}^{-1} \M \Q_{\ell-1} d_- = \Q_{\ell-1}^{-1} \M \Q_{\ell-1} (- \Q_{\ell-1}^{-1} d_-^{CM} \Q_\ell) \\[1ex]
        &= -\Q_{\ell-1}^{-1} \M d_-^{CM} \Q_\ell = -\Q_{\ell-1}^{-1} d_-^{CM} \M \Q_\ell \\[1ex]
        &= -\Q_{\ell-1}^{-1} d_-^{CM} \Q_\ell \Q_\ell^{-1} \M \Q_\ell = d_- \M^\Q.
    \end{align*}

    For $d_+$, we obtain:
    \begin{align*}
        \M^\Q d_+ &= \Q_{\ell+1}^{-1} \M \Q_{\ell+1} \Q_{\ell+1}^{-1} (-d^{CM}_+) \Q_\ell \\[1ex]
        &= \Q_{\ell+1}^{-1} \left( \frac{-1}{1-uy^{CM}_1} \right) d^{* CM}_+ \M \Q_\ell \\[1ex]
        &= \frac{-1}{1-uy_1} \Q_{\ell+1}^{-1} \big( \Q_{\ell+1} y_1 d_+^* \Q_\ell^{-1} \big) \M \Q_\ell \\[1ex]
        &= \frac{-y_1}{1-uy_1} d_+^* \M^\Q.\qedhere
    \end{align*}
\end{proof}

\subsection{Nonsymmetric plethysm and Weyl symmetrization}
\label{nonmac}
We recall the relevant notions useful for our purposes from the recent work of Blasiak et al.\ \cite{nsshuffle} and \cite[Section~7.8]{flaggedllt}.

The \emph{Demazure operator} $\pi_i$ is defined by
\begin{equation}\label{dema}
    \pi_i (f) := \cfrac{(1 - s_i)(x_i f)}{x_i - x_{i+1}}.
\end{equation}

Further, for any \(w \in \mathfrak{S}_N\), given a reduced expression \(s_{i_1}s_{i_2}\cdots s_{i_d}\) of $w$, 
we define \(\pi_w := \pi_{s_{i_1}} \cdots \pi_{s_{i_d}}\).
Moreover, since the operators $\pi_i$ satisfy the braid relations, 
the definition of $\pi_w$ is independent of the choice of reduced expression.

Therefore, we now define a stable version of Weyl symmetrization, 
\(\Weyl_{\ell} \colon \P(\ell) \to \P(\ell-1)\), by
\begin{equation}\label{defweyl}
    \Weyl_{\ell} f[X] := \lim_{N \to \infty} \pi_{x_\ell, x_{\ell+1},\dots,x_N} f[X_N],
\end{equation}
where $\pi_{x_{\ell}, \dots,x_N}$ is the operator $\pi_w$ 
with $w$ being the longest permutation in the subgroup \(\mathfrak{S}_{(1^{\ell-1}; N-\ell+1)} \subseteq \mathfrak{S}_N\), 
which consists of permutations that only act on the last \(N-\ell+1\) indices.
As noted by Blasiak et al.\ in \cite[Section~7.7]{flaggedllt}, the sequence on the right-hand side of \eqref{defweyl} converges strongly to a well-defined limit.

The modified Macdonald polynomial \(\Htild_\mu(X;q,t)\) differs from the Macdonald integral form \(J_\mu(X;q,t)\) only by a plethysm transformation:
\[
    f[X] \mapsto f\left[{X}/{(1-t)}\right].
\]
Blasiak et al.\ \cite{flaggedllt} extended this plethysm transformation to the space \(\P(\ell)\), given by:
\begin{equation}
    \Pisf_\ell(f(x_1, \dots, x_\ell) g[X]) := g\Big[\frac{X}{1-t}\Big] \pol \Big( \frac{f(x_1, \dots, x_\ell)}{\prod_{1 \le i < j \le \ell} (1-t \spa x_i/x_j)} \Big),
\end{equation}
where \(g\) is a symmetric function in the infinite variables \(X = x_1, x_2, \dots\). 
One can readily check that \(\Pisf_\ell\) is a \(\mathbb{Q}(q,t)\)-linear map. 

Moreover, \(\Pisf_\ell\) is invertible. 
This follows from the fact that both \(\flagh_\eta(X_\ell)h_\lambda[X]\) (introduced in \cite[Eq.~(57)]{nsshuffle}) and \(\flagh^\pm_\eta(X_\ell)h_\lambda[(1-t)X]\) (introduced in \cite[Eq.~(59)]{nsshuffle}) 
individually form a basis of \(\P(\ell)\). 
Here, \(\flagh_\eta(X_\ell) := \prod_{i=1}^\ell h_{\eta_i}(x_1, \dots, x_i)\) denotes the flagged homogeneous polynomial, and the pairs \((\eta|\lambda)\) range over \(\pairsr\).
Furthermore, Blasiak et al.\ established the following identity in \cite[Theorem 4.2.1]{nsshuffle}:
\begin{equation}
    \Pisf_\ell \big(\flagh^\pm_\eta(X_\ell)h_\lambda[(1-t)X]\big) = \flagh_\eta(X_\ell)h_\lambda[X].
\end{equation}

As constructed in \cite[Section~7.1]{nsshuffle}, the integral form stable \(\ell\)-nonsymmetric Macdonald polynomials \(\stE_{\eta|\lambda}\) also form a basis of \(\P(\ell)\). 
The modified \(\ell\)-nonsymmetric Macdonald polynomials \(\tE_{\eta|\lambda}\) (introduced in \cite[Section~7.2]{nsshuffle}) are defined by the action of \(\Pisf_\ell\) via \cite[Eq.~(134)]{nsshuffle}:
\begin{equation}\label{pimac}
    \tE_{\eta|\lambda} := \Pisf_\ell \stE_{\eta|\lambda}, \quad \text{or equivalently,} \quad
    \Pisf^{-1}_\ell \tE_{\eta|\lambda} = \stE_{\eta|\lambda}.
\end{equation}
Consequently, the elements \(\tE_{\eta|\lambda}\) individually form another basis of \(\P(\ell)\) as the pairs \((\eta|\lambda)\) range over \(\pairsr\).

We recall the following fundamental identity relating these modified \(\ell\)-nonsymmetric Macdonald polynomials 
and stable Weyl symmetrization, established by Blasiak et al.\ \cite[Proposition 7.2.3]{nsshuffle}:
\begin{equation}
    \Weyl_1 \cdots \Weyl_\ell\spa \tE_{\eta|\lambda}(X;q,t) = \omega\Htild_{(\eta; \lambda)_+}(X;q,t),
\end{equation}
where the notation \((\eta; \lambda)_+\) denotes the partition obtained by concatenating \(\eta\) and \(\lambda\), 
then rearranging the parts into weakly decreasing order.

Finally, we recall a lemma that will be highly useful in our proof.
\begin{lemma}[{\cite[Eq.~(140)]{nsshuffle}}]\label{useful}
    For any \(f \in \P(\ell)\), 
    \[
        \Weyl_1 \cdots \Weyl_\ell \Pisf_\ell f = \Pisf_0 \PP_0^{-1} d_-^\ell \PP_\ell f.
    \]
\end{lemma}

\subsection{Flagged LLT polynomials}\label{flagged}
Flagged LLT polynomials were first introduced by Blasiak et al.\ in \cite{flaggedllt}, 
and the flagged row and column LLT polynomials we work with here, introduced in \cite{nsshuffle}, are simply a special case of this general construction. 
For the explicit combinatorial reformulation of these variants in terms of the general model from \cite{flaggedllt}, 
we refer the reader to \cite[Section~5.3]{nsshuffle}.

The concepts and properties referred to in this subsection are given in \cite[Section~5]{nsshuffle}.
These two flagged LLT polynomials are linked to partial Dyck paths and their corresponding markings. 
Notably, they retain all the core combinatorial properties of ordinary LLT polynomials, 
while adding flagged constraints that bound the allowed values of fillings; 
this extra structure makes them a powerful tool for characterizing nonsymmetric objects in algebraic combinatorics.

\subsubsection{Flagged row LLT polynomial}
\label{flaggedrowllt}
We first recall the fundamental notions of partial Dyck paths and their markings, which underpin the definition of flagged row LLT polynomials. 
Let \(\ell \ge 0\) and \(n > 0\) be integers. We write \(\mathrm{D}_\ell(n)\) for the set of \emph{\(\ell\)-partial Dyck paths}, defined to be lattice paths from \((0,\ell)\) to \((n,n)\) that stay weakly above the diagonal \(y=x\). 
A \textit{marking} of \(\hat{\pi} \in \mathrm{D}_\ell(n)\) is a subset \(\Sigma\) of the outer corner boxes of \(\hat\pi\) — 
that is, the boxes above \(\hat\pi\) formed by an east step followed by a north step. We set \(\Area(\hat\pi) := \Area(\mathsf{N}^\ell \hat\pi)\).

Let \(\mathcal{A}\) be the ``super" alphabet with total order \(1 < \bar{1} < 2 < \bar{2} < \dots\), consisting of \emph{positive letters} \(v\) and \emph{negative letters} \(\bar{v}\) for each \(v \in \ZZ_+\). Let \(|\bar{v}| := v\) for a negative letter \(\bar{v}\) and \(|v| := v\) for a positive letter \(v\).

\begin{defn}[{\cite[Section~5.2]{nsshuffle}}]
A \textit{flagged \((\hat\pi, \Sigma)\)-row super word} is a map \(T \colon [n] \to \mathcal{A}\) satisfying the following conditions:
\begin{enumerate}
    \item For every \((i,j) \in \Sigma\), the pair \(T(i), T(j)\) satisfies \(T(j) < T(i)\) or 
    \(T(i), T(j)\) are equal positive numbers.
    \item The flagged condition: \(T(1) \le 1, T(2) \le 2, \dots, T(\ell) \le \ell\).
\end{enumerate}
\end{defn}
We write \(\FSSYT^\pm(\hat\pi, \Sigma)\) for the set of all such super words. 
Restricting \(T\) to positive entries recovers the set of \textit{flagged \((\hat\pi, \Sigma)\)-row semistandard words}, which we denote by \(\FSSYT(\hat\pi, \Sigma)\).

For \(T \in \FSSYT^\pm(\hat\pi, \Sigma)\), we define two key statistics: \(\inv(T)\), 
the number of attacking inversions (i.e., pairs \((i,j) \in \Area(\hat\pi)\) such that \(T(j) < T(i)\) or 
\(T(i), T(j)\) are equal positive numbers), 
and \(m(T)\), the number of negative letters in \(T\). The associated \textit{signed flagged row LLT polynomial} is then defined by:
\[
    \Growpm_\ell(\hat\pi, \Sigma)(X; t) := \sum_{T \in \FSSYT^\pm(\hat\pi, \Sigma)} t^{\inv(T)} (-t)^{m(T)} \xx^{|T|},
\]
where \(\xx^{|T|} := \prod_{i \in [n]} x_{|T(i)|}\). 




Signed flagged row LLT polynomials can be constructed via successive applications of the operators \(d^\PP_-\) and \(d^\PP_+\) 
from \(\AA_{t,q}\). 
We first introduce the associated operator \(\chi_\ell^\PP(\hat\pi, \Sigma)\), then establish its connection to signed flagged row LLT polynomials.

\begin{defn}[{\cite[Definition 5.5.1]{nsshuffle}}]\label{d-d+p}
Let \(\hat\pi \in \mathrm{D}_\ell(n)\) be a partial Dyck path with marking \(\Sigma\). 
The operator \(\chi_\ell^\PP(\hat\pi, \Sigma)\) is defined by acting on the constant polynomial \(1 \in \P(0)\)
with the following sequence of operators, obtained by traversing \(\hat\pi\) and \(\Sigma\) from the northeast endpoint:
\begin{itemize}
    \item For each marked corner in \(\Sigma\), we apply the operator \(\frac{1}{t-1}[d_-^\PP, -d_+^\PP]\);
    \item For each vertical step of \(\hat\pi\) that is not part of a marked corner, we apply \(d_-^\PP\);
    \item For each horizontal step of \(\hat\pi\) that is not part of a marked corner, we apply \(-d_+^\PP\).
\end{itemize}
\end{defn}

\begin{thm}[{\cite[Theorem 5.5.4]{nsshuffle}}]\label{5.5.4}
For any partial Dyck path \(\hat\pi \in \mathrm{D}_\ell(n)\) and marking \(\Sigma\), the operator \(\chi_\ell^\PP(\hat\pi, \Sigma)\) equals the signed flagged row LLT polynomial:
\[
    \Grow_\ell^\pm(\hat\pi, \Sigma) = \chi_\ell^\PP(\hat\pi, \Sigma).
\]
\end{thm}

If we replace the operators \(-d_+^\PP\) and \(d_-^\PP\) in Definition \ref{d-d+p} with \(d_+\) and \(d_-\), respectively, 
we obtain the element \(\chi_\ell(\hat\pi, \Sigma) \in V_\ell\). 
Thus, by Theorem \ref{5.5.4}, it follows readily that
\begin{equation}\label{e Gcal vs chiPP}
    \Growpm_\ell(\hat{\pi}, \Sigma) = \chi^\PP_\ell(\hat{\pi}, \Sigma) = (-1)^{n} \PP_\ell^{-1} \chi_\ell(\hat{\pi}, \Sigma).
\end{equation}
One can then check that \(\chi(\pi, \Sigma) := d_-^\ell \chi_\ell(\hat{\pi}, \Sigma)\) recovers the symmetric LLT polynomial from \cite[Section~3.5]{mellit}, where we use the notation \(\pi = \mathsf{N}^\ell \hat{\pi}\).

\subsubsection{Flagged column LLT polynomials}
\begin{defn}[{\cite[Section~5.6]{nsshuffle}}]
A \textit{flagged \((\hat\pi, \Sigma)\)-column super word} is a map \(T \colon [n] \to \mathcal{A}\) satisfying the following conditions:
\begin{enumerate}
    \item For every \((i,j) \in \Sigma\), the pair \(T(j), T(i)\) satisfies 
    \(T(j) > T(i)\) or \(T(j), T(i)\) are equal negative numbers;
    \item The flagged condition: \(T(1) \le 1, T(2) \le 2, \dots, T(\ell) \le \ell\).
\end{enumerate}
\end{defn}

We write \(\widetilde{\FSSYT}^\pm(\hat\pi, \Sigma)\) for the set of all such super words, 
and \(\widetilde{\FSSYT}(\hat\pi, \Sigma)\) for the set of positive-valued counterparts.

For \(T \in \widetilde{\FSSYT}^\pm(\hat\pi, \Sigma)\), we define \(\widetilde{\inv}(T)\) as the number of pairs \((i,j) \in \Area(\hat\pi)\) 
such that \(T(j), T(i)\) satisfy \(T(j) > T(i)\) or \(T(j), T(i)\) are equal negative numbers.
The associated \textit{signed flagged column LLT polynomial} is then defined by:
\[
    \Gcolpm_\ell(\hat\pi, \Sigma)(X; t) := \sum_{T \in \widetilde{\FSSYT}^\pm(\hat\pi, \Sigma)} t^{\widetilde{\inv}(T)} (-t)^{-m(T)} \xx^{|T|},
\]
and the \textit{flagged column LLT polynomial} is defined by:
\[
    \Gcol_\ell(\hat\pi, \Sigma)(X; t) := \sum_{T \in \widetilde{\FSSYT}(\hat\pi, \Sigma)} t^{\widetilde{\inv}(T)} \xx^{T}.
\]

Similar to the row case, signed flagged column LLT polynomials can be constructed via successive applications of \(d^\PP_-\) and \(d^\PP_+\). 

\begin{defn}[{\cite[Definition 5.6.3]{nsshuffle}}]
Let \(\hat\pi \in \mathrm{D}_\ell(n)\) be a partial Dyck path with marking \(\Sigma\). 
The operator \(\widetilde{\chi}_\ell^\PP(\hat\pi, \Sigma)\) is defined by acting on the constant polynomial \(1 \in \P(0)\) 
with the following sequence of operators, obtained by traversing \(\hat\pi\) and \(\Sigma\) from the northeast endpoint:
\begin{itemize}
    \item For each marked corner in \(\Sigma\), we apply \(\frac{1}{t^{-1}-1}\bigl(d_-^\PP(-t^{-r}d_+^\PP) - (-t^{-r+1}d_+^\PP) d_-^\PP\bigr)\);
    \item For each vertical step of \(\hat\pi\) that is not part of a marked corner, we apply \(d_-^\PP\);
    \item For each horizontal step of \(\hat\pi\) that is not part of a marked corner, we apply \(-t^{-r}d_+^\PP\).
\end{itemize}
Here \(r = h - v\), with \(h\) and \(v\) being the number of horizontal and vertical steps traversed so far.
\end{defn}

\begin{thm}[{\cite[Theorem 5.6.4]{nsshuffle}}]
For any partial Dyck path \(\hat\pi \in \mathrm{D}_\ell(n)\) and marking \(\Sigma\), 
the operator \(\widetilde{\chi}_\ell^\PP(\hat\pi, \Sigma)\) equals the signed flagged column LLT polynomial:
\[
    \Gcolpm_\ell(\hat\pi, \Sigma)(X; t^{-1}) = \widetilde{\chi}_\ell^\PP(\hat\pi, \Sigma)(X; t).
\]
\end{thm}

\subsubsection{Symmetrization}
We now note that signed and unsigned flagged LLT polynomials are related via the \(\ell\)-nonsymmetric plethysm map.
\begin{thm}[{\cite[Theorem 5.4.1 and Eq.~(94)]{nsshuffle}}]\label{5.4.1}
The \(\ell\)-nonsymmetric plethysm map sends signed flagged LLT polynomials to flagged LLT polynomials. 
Precisely, for any \(\hat\pi \in \mathrm{D}_\ell(n)\) and marking \(\Sigma\) of \(\hat\pi\), we have
\begin{align}
    \Pisf_\ell \big( \Growpm_\ell(\hat\pi, \Sigma) \big) &= \Grow_\ell(\hat\pi, \Sigma), \\
    \label{eq5.4.1} \Pisf_\ell \big( \Gcolpm_\ell(\hat\pi, \Sigma)(X;t^{-1}) \big) &= \Gcol_\ell(\hat\pi, \Sigma)(X;t^{-1}).
\end{align}
\end{thm}

\begin{prop}[{\cite[Proposition 5.7.2]{nsshuffle}}]
    (i) Flagged row LLT polynomials admit the following relation under Weyl symmetrization. 
    Precisely, for any \(\hat\pi \in \mathrm{D}_\ell(n)\) with marking \(\Sigma\), using the notation \(\pi = \mathsf{N}^\ell \hat{\pi}\) (where $\mathsf{N}^\ell$ denotes prepending $\ell$ consecutive north steps), we have
    \begin{equation}\label{weylrow}
        \Weyl_1 \cdots \Weyl_\ell \big( \Grow_\ell(\hat{\pi}, \Sigma) \big) = \omega \spa \chi(\pi, \Sigma) = \omega d^\ell_- \chi_\ell(\hat{\pi}, \Sigma).
    \end{equation}

    (ii) Similarly, flagged column LLT polynomials admit an analogous relation under the same symmetrization:
    \begin{equation}\label{weylcolumn}
        \Weyl_1 \cdots \Weyl_\ell \big( \Gcol_\ell(\hat{\pi}, \Sigma) \big) = \chi(\pi, \Sigma) = d^\ell_- \chi_\ell(\hat{\pi}, \Sigma).
    \end{equation}
\end{prop}

\begin{remark}
    We note that the original statement in \cite[Section~5.7]{nsshuffle} is phrased in terms of the connection to symmetric LLT polynomials. 
    However, by \cite[Remark 5.7.1]{nsshuffle}, the original formulation is equivalent to the one we present here.
\end{remark}

\section{The nonsymmetric compositional Delta theorem}
\label{secnscdt}
In this section, we establish a signed nonsymmetric compositional Delta theorem, whose combinatorial side is given by signed flagged row LLT polynomials. 
By applying the $\ell$-nonsymmetric plethysm map to this identity, we then obtain the nonsymmetric compositional Delta theorem, 
which corresponds to \eqref{eq1}. 
Our proof proceeds by adapting the symmetric approach of D'Adderio and Mellit \cite{proofdelta} 
and reinterpreting it in the language of flagged LLT polynomials, introduced in Subsection~\ref{flagged}. 
Furthermore, by applying the Weyl symmetrization operator $\Weyl_1 \dots \Weyl_\ell$ to our nonsymmetric identity, 
we systematically recover the $\omega$ statement of the original compositional Delta theorem.

\subsection{Signed, nonsymmetric compositional Delta theorem}
We first recall the original compositional Delta theorem.
\begin{thm}[{\cite[Theorem 5.2 and Theorem 5.3]{proofdelta}}]\label{thm compositional delta conjecture}
Let \(\alpha = (\alpha_1, \alpha_2, \dots, \alpha_\ell)\) be a composition of $n-k$ of length $\ell$.  
    \begin{equation}\label{eq compositional delta conjecture}
       \Theta_{e_k} \nabla C_{\alpha}(X;t) = \mathop{\sum_{P\in \LD(n)^{\ast k}}}_{\dcomp(P)=\alpha}q^{\area(P)}t^{\dinv(P)}\xx^{P},
    \end{equation}
    where \(\xx^{P} := \prod_{i \in \ZZ_+} x_i^{\# \text{ of } i\text{'s in } P}\).
\end{thm}

\begin{remark}
From \cite[Proposition 6.2]{proofdelta}, we have the operator identity:
\begin{equation}\label{thetanabla}
    \Theta_{k} \nabla' = \sum_{i=0}^k \e_i^* \nabla' \e_{k-i}^*.
\end{equation}

Throughout the remainder of this work, we will work primarily with $\tau_u^* \nabla' \tau_u^*$, 
since the above identity is equivalent to the coefficient identity (using our notation defined below):
\[
    \sum_{i=0}^k \e_i^* \nabla' \e_{k-i}^* = \left[\tau_u^* \nabla' \tau_u^*\right]_k.
\]
Furthermore, we have the identity \(\Theta_{k} \nabla' f = (-1)^d \Theta_{k} \nabla f\) 
for any homogeneous element \(f \in \sym[X]\) of degree $d$.
As a direct consequence, we obtain the following relation:
\begin{equation}\label{remarksign}
        (-1)^d \left[\tau_u^* \nabla' \tau_u^*\right]_k = (-1)^d \sum_{i=0}^k \e_i^* \nabla' \e_{k-i}^* f = \Theta_{e_k} \nabla  f.
\end{equation}
\end{remark}

Let $n, k$ be positive integers, and let \(\alpha = (\alpha_1, \dots, \alpha_\ell)\) be a composition of $n-k$ of length $\ell$. We define:
\begin{equation}\label{defnalg}
    \alg := (-1)^{n} \PP_\ell^{-1} t^{\ell-|\alpha|} \M^\Q_k \left( y_1^{\alpha_1-1} \dots y_\ell^{\alpha_\ell-1} d_+^\ell (1)\right).
\end{equation}
We adopt the following standard notation: for any operator $f$ depending on the formal parameter $u$, 
we write $f_k$ to denote the coefficient operator $(-1)^k \left<u^k\right> f$.

With this notation, we have \(\M^\Q_k := (\M^\Q)_k\), where \(\M^\Q := \Q_\ell^{-1} \M \Q_\ell\). 
Analogously, we write \(\M_k := (\M)_k\).

Next, we recall the $\zeta$ map from D'Adderio, Iraci, and Vanden Wyngaerd \cite[Proposition 6.7]{theta}, which is an extension of Haglund's classical result \cite[Theorem 3.15]{hag08}. We write $\pi'$ for the image of a decorated Dyck path $\pi$ under this $\zeta$ map. Notably, if the diagonal composition of $\pi$ is the composition $\alpha$ with length $\ell$, then $\pi'$ starts with $\ell$ consecutive north steps. Let $\hat{\pi}' \in \mathrm{D}_\ell(n)$ be the partial Dyck path obtained by removing these first $\ell$ north steps, i.e., $\pi' = \mathsf{N}^\ell \hat{\pi}'$. 
We define the marking $\Sigma_\pi$ of $\hat{\pi}'$ to be the set of corner boxes corresponding to pairs of consecutive north steps along the vertical runs of $\pi$.

We are now in a position to state our first main result. It provides a purely combinatorial characterization of the algebraic element $\alg$ defined in \eqref{defnalg}, utilizing signed flagged row LLT polynomials.

\begin{thm}\label{q ns delta unmod}
For positive integers $n, k$ and a strict composition \(\alpha = (\alpha_1, \dots, \alpha_\ell)\) of size $n-k$,
\begin{equation} \label{e ns delta unmod}
    \alg = \mathop{\sum_{(\pi, dr)\in \mathrm{D}(\alpha)^{\ast k}}} q^{\area(\pi, dr)} \Growpm_\ell(\hat{\pi}', \Sigma_\pi),
\end{equation}
where $\area(\pi, dr)$ and $\Growpm_\ell(\hat{\pi}', \Sigma_\pi)$ are defined in Section~\ref{combinatorialobj} and Subsection~\ref{flaggedrowllt}, respectively.
\end{thm}

\begin{proof}
First, we note that by identity \eqref{e Gcal vs chiPP}, equation \eqref{e ns delta unmod} is equivalent to the following identity:
    \begin{equation}\label{micheleeq}
        t^{\ell-|\alpha|} \M^\Q_k \left( y_1^{\alpha_1-1} \dots y_\ell^{\alpha_\ell-1} d_+^\ell (1)\right) = 
        \mathop{\sum_{(\pi, dr)\in \mathrm{D}(\alpha)^{\ast k}}} q^{\area(\pi, dr)} \chi_\ell(\hat{\pi}', \Sigma_\pi).
    \end{equation}
To prove \eqref{micheleeq}, we first recall how D'Adderio and Mellit \cite{proofdelta} and D'Adderio et al.\ \cite{theta} proved the original compositional Delta theorem.

In \cite[Theorem 6.5]{proofdelta}, D'Adderio and Mellit showed that
    \begin{equation}\label{eq:mellit20}
        t^{\ell-|\alpha|} \M_k \left( (y^{CM}_1)^{\alpha_1-1} \dots (y^{CM}_\ell)^{\alpha_\ell-1} (d^{CM}_+)^\ell (1) \right) = M_\alpha^{\ast k}.
    \end{equation}
Here, $M_\alpha^{\ast k}$ is the element constructed by successive actions of $d^{CM}_-$ and $d^{CM}_+$ on $1$ originally established in \cite{theta}.

We define $\chi^{CM}_\ell (\hat{\pi}', \Sigma_\pi)$ to be the expression obtained from Definition \ref{d-d+p} 
by replacing $-d^\PP_+$ and $d^\PP_-$ with $d^{CM}_+$ and $d^{CM}_-$, respectively. 
One can readily check that every term in $\chi^{CM}_\ell(\hat{\pi}', \Sigma_\pi)$ 
is an element obtained by acting on $1$ with $2n-\ell$ operators from $\{d_-^{CM}, d_+^{CM}\}$. In particular, since \(\Q^{-1}_0 (1) = 1\), we have
    \[
        \chi^{CM}_\ell(\hat{\pi}', \Sigma_\pi) = (-1)^{2n-\ell} \Q_\ell \chi_\ell(\hat{\pi}', \Sigma_\pi).
    \]

Furthermore, D'Adderio et al.\ established the following combinatorial expansion in \cite[Theorem 6.22]{theta}:
    \begin{equation}\label{combd-}
        M_\alpha^{\ast k} = \mathop{\sum_{(\pi, dr)\in \mathrm{D}(\alpha)^{\ast k}}} q^{\area(\pi, dr)} \chi^{CM}_\ell(\hat{\pi}', \Sigma_\pi).
    \end{equation}

Combining \eqref{eq:mellit20} and \eqref{combd-}, we immediately obtain:
    \begin{equation}\label{eq:mellit20combd-}
        t^{\ell-|\alpha|} \M_k \left( (y^{CM}_1)^{\alpha_1-1} \dots (y^{CM}_\ell)^{\alpha_\ell-1} (d^{CM}_+)^\ell (1) \right) = 
        \mathop{\sum_{(\pi, dr)\in \mathrm{D}(\alpha)^{\ast k}}} q^{\area(\pi, dr)} \chi^{CM}_\ell(\hat{\pi}', \Sigma_\pi).
    \end{equation}

Finally, we rewrite identity \eqref{eq:mellit20combd-} in terms of our operators $d_-, d_+, y_i$, using the correspondence from \eqref{CMM}:
    \[
        (-1)^\ell t^{\ell-|\alpha|} \M_k \Q_\ell \left( y_1^{\alpha_1-1} \dots y_\ell^{\alpha_\ell-1} d_+^\ell (1) \right) = 
        (-1)^{2n-\ell} \Q_\ell \mathop{\sum_{(\pi, dr)\in \mathrm{D}(\alpha)^{\ast k}}} q^{\area(\pi, dr)} \chi_\ell(\hat{\pi}', \Sigma_\pi).
    \]

Recalling that \(\M^\Q_k = \Q_\ell^{-1} \M_k \Q_\ell\), we can cancel the $\Q_\ell$ factors from both sides, which yields \eqref{micheleeq} and completes the proof.
\end{proof}

\subsection{The nonsymmetric compositional Delta theorem}
By applying the $\ell$-nonsymmetric plethysm map $\Pisf_\ell$ to both sides of Theorem \ref{q ns delta unmod} and noting Theorem \ref{5.4.1}, we immediately obtain our second main result, which provides an unsigned formulation. This expresses the action of $\Pisf_\ell \spa \alg$ strictly in terms of flagged row LLT polynomials.

\begin{thm} \label{q ns delta mod}
With the notation above, we have:
\begin{equation} \label{e ns delta mod}
    \Pisf_\ell \spa \alg
    = \mathop{\sum_{(\pi, dr)\in \mathrm{D}(\alpha)^{\ast k}}} q^{\area(\pi, dr)} \Grow_\ell(\hat{\pi}', \Sigma_\pi).
\end{equation}
\end{thm}

We can reformulate Theorem \ref{q ns delta mod} into a form strictly parallel to the original statement of the compositional Delta theorem by expanding the flagged row LLT polynomial $\Grow_\ell(\hat{\pi}', \Sigma_\pi)$.

\begin{defn}[{\cite[Section 6.3]{nsshuffle}}]
    Let \(\pi \in \mathrm{D}(\alpha)^{\ast k}\). A \emph{flagged weak word parking function} on $\pi$ is a labeling $w$ of the north steps of $\pi$ by positive integers satisfying:
    \begin{itemize}
        \item The labels are weakly decreasing upward along each vertical run of $\pi$;
        \item The label of the first $j$ north steps lying on the main diagonal is at most $j$, for \(1 \le j \le \ell\).
    \end{itemize}

We denote the set of all flagged weak word parking functions on $\pi$ by $\FWPF'_\alpha(\pi)$.

For any \(w \in \FWPF'_\alpha(\pi)\), let $w_i$ denote the label of the north step in the $i$-th row. The \emph{weak temporary diagonal inversion statistic} is defined as:
\[
    \dinv'(\pi, w) := \#\left\{ (i,j) \mid 1 \le i < j \le n, \ i \text{ attacks } j, \ \text{and } w_i \ge w_j \right\}.
\]
\end{defn}

This natural combinatorial description directly yields the identity stated in \eqref{eq1}.
\begin{cor}\label{c ns delta mod}
With the notation above,
\[
    \Pisf_\ell \spa \alg
    = \sum_{\pi \in \mathrm{D}(\alpha)^{\ast k}} \sum_{w \in \FWPF'_\alpha(\pi)} q^{\area(\pi, dr)} t^{\dinv'(\pi, w)} \,\xx^{\content(w)},
\]
where \(\xx^{\content(w)} := \prod_{i \in \ZZ_+} x_i^{\# \text{ of } i\text{'s in } w}\).
\end{cor}

To demonstrate the strength of our nonsymmetric framework, we now show that applying Weyl symmetrization to the nonsymmetric compositional Delta theorem recovers the $\omega$ statement of the original compositional Delta conjecture.

\begin{prop}\label{prop:recover_omega}
Applying the Weyl symmetrization operator $\Weyl_1 \dots \Weyl_\ell$ to both sides of \eqref{e ns delta mod}
recovers the $\omega$ statement of the original compositional Delta theorem.
\end{prop}

\begin{proof}
We first recall the following identity for $C_\alpha$ from Mellit \cite[Eq.~(22)]{mellit}. 
Recall that $\overline{\omega}$ is the operator acting on $\sym[X]$ that maps \(F[X;q,t] \mapsto F[-X;q^{-1},t^{-1}]\).
\begin{equation}\label{eq:Calpha}
        (-1)^{|\alpha|} C_\alpha(X;t) = t^{\ell-|\alpha|} \overline{\omega} 
        \left((d_-)^\ell y_1^{\alpha_1-1} \dots y_\ell^{\alpha_\ell-1} (d_+)^\ell (1)\right).
\end{equation}

We also note the identity:
\begin{equation}\label{pi0p0}
    \Pisf_0 \PP_0^{-1} (f[X]) = f[-X] = (-1)^{d} \omega f[X],
\end{equation}
for any homogeneous symmetric function $f$ of degree $d$.

Furthermore, \eqref{weylrow} provides:
\[
    \Weyl_1 \dots \Weyl_\ell \big(\Grow_\ell(\hat{\pi}', \Sigma_\pi)\big) = \omega \spa d_-^\ell \chi_\ell(\hat{\pi}', \Sigma_\pi).
\]

Hence, applying the Weyl symmetrization operator $\Weyl_1 \dots \Weyl_\ell$ to the left-hand side of \eqref{e ns delta mod} 
yields the following aligned derivation:
\begin{align}
        &\mathrel{\phantom{=}} \Weyl_1 \dots \Weyl_\ell \Pisf_\ell \spa \alg \nonumber \\[1ex]
        &= \Weyl_1 \dots \Weyl_\ell \Pisf_\ell (-1)^{n} t^{\ell-|\alpha|} \PP_\ell^{-1} \M^\Q_k \left( y_1^{\alpha_1-1} \dots y_\ell^{\alpha_\ell-1} d_+^\ell (1)\right) && \text{(by \eqref{defnalg})} \nonumber \\[1ex]
        &= \Pisf_0 \PP_0^{-1} d_-^\ell \PP_\ell (-1)^{n} t^{\ell-|\alpha|} \PP_\ell^{-1} \M^\Q_k \left( y_1^{\alpha_1-1} \dots y_\ell^{\alpha_\ell-1} d_+^\ell (1)\right) && \text{(by Lemma \ref{useful})} \nonumber \\[1ex]
        &= (-1)^n \omega (-1)^n d_-^\ell t^{\ell-|\alpha|} \M^\Q_k \left( y_1^{\alpha_1-1} \dots y_\ell^{\alpha_\ell-1} d_+^\ell (1)\right) && \text{(by \eqref{pi0p0})} \nonumber \\[1ex]
        &= \omega \M^\Q_k t^{\ell-|\alpha|} \left(d_-^\ell y_1^{\alpha_1-1} \dots y_\ell^{\alpha_\ell-1} d_+^\ell (1)\right) && \text{(by \eqref{eqprenewM})} \nonumber \\[1ex]
        &= \omega \left[\tau_u^{*} \nabla' \overline{\omega} (\tau_{qtu}^{*})^{-1}\right]_k t^{\ell-|\alpha|} \left(d_-^\ell y_1^{\alpha_1-1} \dots y_\ell^{\alpha_\ell-1} d_+^\ell (1)\right) && \text{(by \eqref{defnM})} \nonumber \\[1ex]
        &= \omega \left[\tau_u^{*} \nabla' \tau_{u}^{*} \overline{\omega} \right]_k t^{\ell-|\alpha|} \left(d_-^\ell y_1^{\alpha_1-1} \dots y_\ell^{\alpha_\ell-1} d_+^\ell (1)\right) \nonumber \\[1ex]
        &= \omega (-1)^{|\alpha|} \left[\tau_u^{*} \nabla' \tau_{u}^{*} \overline{\omega} \right]_k (-1)^{|\alpha|} t^{\ell-|\alpha|} \left(d_-^\ell y_1^{\alpha_1-1} \dots y_\ell^{\alpha_\ell-1} d_+^\ell (1)\right) \nonumber \\[1ex]
        &= \omega (-1)^{|\alpha|} \left[\tau_u^{*} \nabla' \tau_{u}^{*}\right]_k C_\alpha(X;t) && \text{(by \eqref{eq:Calpha})} \nonumber \\[1ex]
        &= \omega \Theta_{e_k} \nabla C_\alpha(X;t). && \text{(by \eqref{remarksign})} \label{omegaleft}
\end{align}

Similarly, applying the Weyl symmetrization operator to the right-hand side of \eqref{e ns delta mod} yields:
\begin{align}
        &\mathrel{\phantom{=}} \Weyl_1 \dots \Weyl_\ell  \mathop{\sum_{(\pi, dr)\in \mathrm{D}(\alpha)^{\ast k}}} q^{\area(\pi, dr)} \Grow_\ell(\hat{\pi}', \Sigma_\pi) \nonumber \\[1ex]
        &= \mathop{\sum_{(\pi, dr)\in \mathrm{D}(\alpha)^{\ast k}}} q^{\area(\pi, dr)} \Weyl_1 \dots \Weyl_\ell \Grow_\ell(\hat{\pi}', \Sigma_\pi) \nonumber \\[1ex]
        &= \mathop{\sum_{(\pi, dr)\in \mathrm{D}(\alpha)^{\ast k}}} q^{\area(\pi, dr)} \omega \spa d_-^\ell \chi_\ell(\hat{\pi}', \Sigma_\pi) && \text{(by \eqref{weylrow})} \nonumber \\[1ex]
        &= \omega d_-^\ell \mathop{\sum_{(\pi, dr)\in \mathrm{D}(\alpha)^{\ast k}}} q^{\area(\pi, dr)} \chi_\ell(\hat{\pi}', \Sigma_\pi) \nonumber \\[1ex]
        &= \omega \mathop{\sum_{P\in \LD(n)^{\ast k}}}_{\dcomp(P)=\alpha}q^{\area(P)} t^{\dinv(P)} \xx^{P}. && \text{(by \cite[Remark 6.21]{theta})} \label{omegaright}
\end{align}

Comparing the two identities \eqref{omegaleft} and \eqref{omegaright}, we successfully recover the $\omega$ statement of Theorem \ref{thm compositional delta conjecture}.
\end{proof}

\section{New nonsymmetric formulation of the compositional Delta theorem}
\label{secnnnf}
In this section, we introduce the nonsymmetric variants of the $\nabla$ operator developed by Blasiak et al.\ \cite{nsshuffle}, 
namely the signed nonsymmetric nabla operator $\umnab$ and the nonsymmetric nabla operator $\modnab$. 
We also present the natural nonsymmetric generalizations of the operator $\tau_{u,\ell}^*$, 
denoted by $\overline{\tau_{u,\ell}^{*}}$ and $\underline{\tau_{u,\ell}^{*}}$.

We first provide an equivalent characterization of the element $\alg$ in terms of $\overline{\tau_{u,\ell}^{*}}$ and $\umnab$. 
Then, utilizing the involution $\barwb^\PP$, we establish a second equivalent characterization of $\alg$ in terms of $\modnab$ and $\underline{\tau_{u,\ell}^{*}}$.
This enables us to deduce an alternative version of the nonsymmetric compositional Delta theorem, corresponding to the identity \eqref{eq2} stated in the introduction.

Most notably, we show that applying the Weyl symmetrization operator $\Weyl_1 \dots \Weyl_\ell$ to this nonsymmetric identity 
systematically recovers the original compositional Delta theorem.

\subsection{An operator identity for \texorpdfstring{$\M$}{M}}
In \cite[p.~19]{mellit}, Mellit introduced an endomorphism \(\Mnab \colon V_\ell \to V_\ell\) for all \(\ell \ge 0\), 
which is uniquely determined by the following two properties:

\begin{enumerate}
    \item \(\Mnab(1) = 1\), where \(1 \in V_0\).
    \item \(\Mnab \circ L = N(L) \circ \Mnab\) for all \(L \in \AA_{t,q}\).
\end{enumerate}

Here, $N$ denotes the unique $\mathbb{Q}(q,t)$-algebra endomorphism \(N \colon \AA_{t,q} \to \AA_{t,q}\) 
from \cite[Proposition 3.16]{mellit}, determined by the mapping:
\[
    e_\ell \mapsto e_\ell, \quad T_i \mapsto T_i, \quad d_-  \mapsto d_- , \quad d_+^* \mapsto d_+^*, \quad d_+ \mapsto -(qt)^{-1} z_1 d_+.
\]

Notably, Mellit showed in \cite[p.~19]{mellit} that when restricted to \(V_0 = \sym[X]\), 
this extended $\Mnab$ coincides with the $\nabla'$ operator on symmetric functions defined earlier.
\begin{remark}
    Recall the commutation relation:
    \begin{equation}\label{mnab-1}
        \Mnab^{-1} d_- = d_- \Mnab^{-1}.
    \end{equation}
\end{remark}

\begin{prop}[{\cite[Proposition 3.15]{mellit}}]
There exists an antilinear involution $\bar\omega$ on the Dyck path algebra $\AA_t$, determined by
\[
    e_\ell \mapsto e_\ell, \quad T_i \mapsto T_i^{-1}, \quad d_-  \mapsto d_- , \quad d_+ \mapsto -t^{-\ell} d_+,
\]
where $d_+$ is the operator mapping from $V_\ell$. 

When restricted to $V_0$, this involution coincides with the involution $\overline{\omega}$ defined earlier.
\end{prop}

\begin{thm}[{\cite[Theorem 3.12]{mellit}}]\label{8.4.2}
There exists a left $\AA_t$-module isomorphism:
\[
    \AA_t e_0 \xrightarrow{\ \cong\ }  \bigoplus_{\ell=0}^{\infty} y_1 \dots y_\ell V_\ell, 
\]
determined by mapping \(e_0 \mapsto 1 \in V_0\).
\end{thm}

Motivated by the above two results, Blasiak et al.\ \cite[p.~39]{nsshuffle} 
defined an involution $\barwb$ on \(\bigoplus_{\ell=0}^{\infty} y_1 \dots y_\ell V_\ell\), which is uniquely determined by the following two properties:
\begin{enumerate}
    \item \(\barwb(1) = 1\), where \(1 \in V_0\).
    \item \(\barwb \circ  L = \bar\omega(L) \circ \barwb\) for all \(L \in \AA_{t}\).
\end{enumerate}

We also recall the following relation established in \cite{nsshuffle}.
\begin{prop}[{\cite[Lemma 8.4.4]{nsshuffle}}]\label{8.4.4}
   Let \(\N^\Q := \Q_\ell^{-1} \N \Q_\ell \colon y_1\dots y_\ell V_\ell \to y_1\dots y_\ell V_\ell\). The following identity holds on \((y_1 \dots y_\ell V_\ell)_{\ell \ge 0}\):
    \[
        \N^\Q = \Mnab\barwb = \barwb \Mnab^{-1}.
    \]
\end{prop}

We now present a new operator identity for $\M^\Q$. We first note that
\[
    \M = \tau_{u,\ell}^* \N (\tau_{qtu,\ell}^*)^{-1} \quad \text{and} \quad \Q_\ell^{-1} \tau_{u,\ell}^* \Q_\ell = \tau_{u,\ell}^*.
\]

It follows that
\begin{equation}\label{defM}
    \M^\Q = \tau_{u,\ell}^* \N^\Q \bigl(\tau_{qtu,\ell}^*\bigr)^{-1} = \tau_{u,\ell}^* \Mnab\barwb \bigl(\tau_{qtu,\ell}^*\bigr)^{-1}.
\end{equation}







\begin{prop}\label{newM}
    We have the following operator identity on \((y_1 \dots y_\ell V_\ell)_{\ell \ge 0}\):
    \begin{equation}
        \M^\Q = \tau_{u,\ell}^* \Mnab \tau_{u,\ell}^* \barwb.
    \end{equation}
\end{prop}

\begin{proof}
    Since the operators $d_+$, $d_-$, and $T_i$ generate $\AA_t$, 
    it follows from Theorem \ref{8.4.2} that they also generate the space \(\bigoplus_{r \ge 0} y_1\dots y_\ell V_\ell\) from the element \(1 \in V_0\). 
    Hence, it suffices to verify that $\M^\Q$ and $\tau_{u,\ell}^* \Mnab \tau_{u,\ell}^* \barwb$ satisfy identical intertwining relations.

    We first verify the intertwining relation for $T_i$. Recall that:
    \[
        \barwb T_i = T_i^{-1} \barwb, \quad \tau_{u,\ell}^* T_i = T_i \tau_{u,\ell}^*, \quad \Mnab T_i = T_i \Mnab.
    \]
    A direct computation gives:
    \[
        \tau_{u,\ell}^* \Mnab \tau_{u,\ell}^* \barwb T_i = \tau_{u,\ell}^* \Mnab \tau_{u,\ell}^* T_i^{-1} \barwb = 
        \tau_{u,\ell}^* \Mnab T_i^{-1} \tau_{u,\ell}^* \barwb = \tau_{u,\ell}^* T_i^{-1} \Mnab \tau_{u,\ell}^* \barwb = T_i^{-1} \tau_{u,\ell}^* \Mnab \tau_{u,\ell}^* \barwb.
    \]
   
    For $d_-$, we recall the relations:
    \[
        \barwb d_- = d_- \barwb, \quad \tau_{u,\ell-1}^* d_- = d_- \tau_{u,\ell}^*,\quad \Mnab d_- = d_- \Mnab,
    \]
    from which we compute:
    \[
        \tau_{u,\ell-1}^* \Mnab \tau_{u,\ell-1}^* \barwb d_- = \tau_{u,\ell-1}^* \Mnab \tau_{u,\ell-1}^* d_-\barwb = \tau_{u,\ell-1}^* \Mnab d_- \tau_{u,\ell}^* \barwb 
        = \tau_{u,\ell-1}^* d_- \Mnab \tau_{u,\ell}^* \barwb = d_- \tau_{u,\ell}^* \Mnab \tau_{u,\ell}^* \barwb.
    \]

    The verification for $d_+$ is analogous. Recall that:
    \[
        \barwb d_+ = -t^{-\ell} d_+ \barwb, \quad \tau_{u,\ell+1}^* d_+ = d_+ \tau_{u,\ell}^*, \quad \Mnab d_+ = (-qt)^{-1} z_1 d_+ \Mnab,
    \]
    \[
        \tau_{u,\ell}^* y_i  = y_i \tau_{u,\ell}^* ,\quad  
        d^{*}_+ \tau^*_{u,\ell}  = (1-u y_1) \tau^*_{u,\ell+1} d^*_+,
    \]
    Combining these with the relation \(z_1 d_+ = -q t^{\ell+1} y_1 d_+^*\) stated in \eqref{relation}, we obtain:
    \begin{align*}
        \tau_{u,\ell+1}^* \Mnab \tau_{u,\ell+1}^* \barwb d_+ &= \tau_{u,\ell+1}^* \Mnab \tau_{u,\ell+1}^* \left(-t^{-\ell} d_+\right) \barwb = \tau_{u,\ell+1}^* \Mnab \left(-t^{-\ell} d_+\right) \tau_{u,\ell}^* \barwb \\
        &= \tau_{u,\ell+1}^* \left((-qt)^{-1} (-t^{-\ell}) z_1 d_+\right) \Mnab \tau_{u,\ell}^* \barwb = \tau_{u,\ell+1}^* \left(\cfrac{1}{qt^{\ell+1}} z_1 d_+\right) \Mnab \tau_{u,\ell}^* \barwb\\
        &= - \tau_{u,\ell+1}^* y_1 d_+^* \Mnab \tau_{u,\ell}^* \barwb = 
        - y_1 \tau_{u,\ell+1}^* d_+^* \Mnab \tau_{u,\ell}^* \barwb = \cfrac{-y_1}{1-uy_1} d_+^{*} \tau_{u,\ell}^{*} \Mnab \tau_{u,\ell}^* \barwb.
    \end{align*}

    Comparing these intertwining relations with those of $\M^\Q$ established in Proposition \ref{prenewM}, the proposition follows.
\end{proof}

\begin{cor}\label{imcor}
The operator $\M^\Q$ admits the following equivalent factorization:
\[
    \M^\Q = \barwb \left(\tau_{qtu,\ell}^*\right)^{-1} \Mnab^{-1} \left(\tau_{qtu,\ell}^*\right)^{-1}.
\]
In particular, 
\begin{equation}\label{newmqk}
    \M^\Q_k = \barwb \left[\left(\tau_{qtu,\ell}^*\right)^{-1} \Mnab^{-1} \left(\tau_{qtu,\ell}^*\right)^{-1}\right]_k.
\end{equation}
\end{cor}

\begin{proof}
It is clear that both $\tau_{u,\ell}^*$ and $\Mnab$ are invertible. 
Combining this with Definition \eqref{defM} of $\M^\Q$ and Proposition \ref{newM}, we first deduce the commutation relation:
\[
    \barwb \left(\tau_{qtu,\ell}^*\right)^{-1} = \tau_{u,\ell}^* \barwb.
\]

Recall that the original definition of $\M^\Q$ is given by:
\[
    \M^\Q = \tau_{u,\ell}^* \Mnab \barwb \left(\tau_{qtu,\ell}^*\right)^{-1}.
\]

Invoking the identity \(\Mnab \barwb = \barwb \Mnab^{-1}\) from Blasiak et al.\ \cite[Lemma 8.4.4]{nsshuffle}, 
we can rewrite the product $\tau_u^* \Mnab \barwb$ as:
\[
    \tau_{u,\ell}^* \Mnab \barwb = \tau_{u,\ell}^* \barwb \Mnab^{-1} = \barwb \left(\tau_{qtu,\ell}^*\right)^{-1} \Mnab^{-1},
\]
where the second equality follows directly from the commutation relation we just derived. 
Substituting this back into the original definition of $\M^\Q$ immediately yields the desired factorization, 
completing the proof.
\end{proof}

\subsection{New formulation of the nonsymmetric compositional Delta theorem}
The operator $\Mnab$ is originally defined on $V_\ell$. 
Conjugating $\Mnab$ via the isomorphism \eqref{isopr} yields an operator acting on $\P(\ell)$. 
We denote this conjugated operator by $\Mnab^\PP$, which is explicitly given by:
\[
    \Mnab^\PP := \PP_\ell^{-1} \Mnab \PP_\ell.
\]

Furthermore, utilizing the $\ell$-nonsymmetric plethysm map $\Pisf_\ell$, 
we can also conjugate $\Mnab^\PP$ to define another operator acting on $\P(\ell)$. This operator is given by:
\[
    \Pisf_\ell \Mnab^\PP \Pisf^{-1}_\ell = \Pisf_\ell \PP_\ell^{-1} \Mnab \PP_\ell \Pisf^{-1}_\ell.
\]

In fact, these two operators admit a highly elegant explicit characterization on $\P(\ell)$, 
as established in \cite[Theorem 8.3.1]{nsshuffle}.

Recall that in Subsection~\ref{nonmac}, we introduced two bases for $\P(\ell)$: the integral form $\stE_{\eta|\lambda}$ and the modified form $\tE_{\eta|\lambda}$, corresponding to the symmetric Macdonald bases $J_\mu(X;q,t)$ and $\Htild_\mu(X;q,t)$, respectively.

Using these two bases, Blasiak et al.\ defined two variants of the nabla operator on $\P(\ell)$ in \cite[Section 8.1]{nsshuffle}: 
the signed nonsymmetric nabla operator $\umnab$ and the nonsymmetric nabla operator $\modnab$, 
which are given by:
\begin{equation}\label{141}
    \umnab \stE_{\eta| \lambda}(X ;q,t) := (-1)^{|\mu|} T_\mu \stE_{\eta| \lambda}(X ;q,t), \qquad \text{where } \mu = (\eta; \lambda)_+,
\end{equation}
\begin{equation}\label{142}
    \modnab \tE_{\eta| \lambda}(X ;q,t) := (-1)^{|\mu|} T_\mu \tE_{\eta| \lambda}(X ;q,t), \qquad \text{where } \mu = (\eta; \lambda)_+.
\end{equation}

\begin{thm}[{\cite[Theorem 8.3.1]{nsshuffle}}]
We have the following operator identities on $\P(\ell)$:
\begin{equation}\label{eq831}
    \Mnab^\PP = \umnab, \qquad \Pisf_\ell \Mnab^\PP \Pisf^{-1}_\ell = \modnab.
\end{equation}
\end{thm}

Analogous to the construction of $\Mnab^\PP$ and $\Pisf_\ell \Mnab^\PP \Pisf^{-1}_\ell$, 
we can similarly define two conjugated versions of $\tau_{u,\ell}^*$ acting on $\P(\ell)$, 
via the isomorphism \eqref{isopr} and the $\ell$-nonsymmetric plethysm map $\Pisf_\ell$, respectively:
\begin{equation}\label{defnptau}
    \overline{\tau_{u,\ell}^*} := \PP_\ell^{-1} \tau_{u,\ell}^* \PP_\ell.
\end{equation}
\begin{equation}\label{defnpiptau}
    \underline{\tau_{u,\ell}^*} := \Pisf_\ell \overline{\tau_{u,\ell}^*} \Pisf^{-1}_\ell 
    = \Pisf_\ell \PP_\ell^{-1} \tau_{u,\ell}^* \PP_\ell \Pisf^{-1}_\ell.
\end{equation}

Moreover, since \eqref{isopr} is an isomorphism, we can express \eqref{defnptau} in the following more explicit form:
\[
    \overline{\tau_{u,\ell}^*} = \sum_{n=0}^\infty (-u)^n \e_n\left[\frac{X}{q-1}\right].
\]
Here, the notation \(\e_n\left[\frac{X}{q-1}\right]\) 
means that for any \(f \in \P(\ell)\), the operator acts as:
\[
    \e_n\left[\frac{X}{q-1}\right] f := e_n\left[\frac{X}{q-1}\right] \cdot f.
\]

\begin{remark}
    However, because $\Pisf_\ell$ is not an algebra homomorphism, 
    we do not currently possess a simple closed-form representation for $\underline{\tau_{u,\ell}^*}$.
\end{remark}

We are now prepared to present our next major result. Utilizing the newly defined operators, we systematically re-factor the algebraic element $\alg$.

\begin{thm}\label{newform}
With the above notation for $\umnab$ and $\overline{\tau_{u,\ell}^*}$, we can rewrite the element $\alg$ defined in \eqref{defnalg} as follows:
\[
    \alg = (-1)^{k} (-t)^{\ell-|\alpha|} \barwb^{\PP} \left[\overline{\left(\tau_{qtu,\ell}^*\right)}^{-1} \left(\umnab\right)^{-1} 
    \overline{\left(\tau_{qtu,\ell}^*\right)}^{-1}\right]_k \left( x_1^{\alpha_1} \dots x_\ell^{\alpha_\ell}\right),
\]
where we define \(\barwb^\PP := \PP_\ell^{-1} \barwb \PP_\ell\).
\end{thm}

\begin{proof}
    By Corollary \ref{imcor}, together with the direct computation that \(d_+^\ell \cdot 1 = (-1)^\ell y_1 \dots y_\ell\), we can rewrite $\alg$ systematically as follows:
\begin{align*}
    & \mathrel{\phantom{=}} (-1)^{n} \PP_\ell^{-1} t^{\ell-|\alpha|} \M^\Q_k \left( y_1^{\alpha_1-1} \dots y_\ell^{\alpha_\ell-1} d_+^\ell (1)\right) \\[1ex]
    &= (-1)^{n} \PP_\ell^{-1} t^{\ell-|\alpha|} \barwb \left[\left(\tau_{qtu,\ell}^*\right)^{-1} \Mnab^{-1} \left(\tau_{qtu,\ell}^*\right)^{-1}\right]_k \left( y_1^{\alpha_1-1} \dots y_\ell^{\alpha_\ell-1} d_+^\ell (1)\right) && \text{(by \eqref{newmqk})} \\[1ex]
    &= (-1)^{n+\ell} t^{\ell-|\alpha|} \PP_\ell^{-1} \barwb \left[\left(\tau_{qtu,\ell}^*\right)^{-1} \Mnab^{-1} \left(\tau_{qtu,\ell}^*\right)^{-1}\right]_k \PP_\ell \left( x_1^{\alpha_1} \dots x_\ell^{\alpha_\ell}\right) \\[1ex]
    &= (-1)^{n+\ell} t^{\ell-|\alpha|} \PP_\ell^{-1} \barwb \left[\PP_{\ell} \PP_\ell^{-1} \left(\tau_{qtu,\ell}^*\right)^{-1} 
    \PP_{\ell} \PP_\ell^{-1} \Mnab^{-1} \PP_{\ell} \PP_\ell^{-1} \left(\tau_{qtu,\ell}^*\right)^{-1} \PP_\ell\right]_k \left( x_1^{\alpha_1} \dots x_\ell^{\alpha_\ell}\right) \\[1ex]
    &= (-1)^{k} (-t)^{\ell-|\alpha|} \barwb^{\PP} \left[\overline{\left(\tau_{qtu,\ell}^*\right)}^{-1} \left(\Mnab^\PP\right)^{-1} 
    \overline{\left(\tau_{qtu,\ell}^*\right)}^{-1}\right]_k \left( x_1^{\alpha_1} \dots x_\ell^{\alpha_\ell}\right) && \text{(by \eqref{defnptau})} \\[1ex]
    &= (-1)^{k} (-t)^{\ell-|\alpha|} \barwb^{\PP} \left[\overline{\left(\tau_{qtu,\ell}^*\right)}^{-1} \left(\umnab\right)^{-1} 
    \overline{\left(\tau_{qtu,\ell}^*\right)}^{-1}\right]_k \left( x_1^{\alpha_1} \dots x_\ell^{\alpha_\ell}\right) && \text{(by \eqref{eq831})}.
\end{align*}
\end{proof}

With this new factorization, Theorem \ref{q ns delta unmod} can be restated as an identity between operators and signed flagged row LLT polynomials:
\begin{multline} \label{e ns shuffle unmod 2}
    (-1)^{k} (-t)^{\ell-|\alpha|} \barwb^{\PP} \left[\overline{\left(\tau_{qtu,\ell}^*\right)}^{-1} \left(\umnab\right)^{-1} \overline{\left(\tau_{qtu,\ell}^*\right)}^{-1}\right]_k
    \left( x_1^{\alpha_1} \dots x_\ell^{\alpha_\ell}\right)\\
     = \mathop{\sum_{(\pi, dr)\in \mathrm{D}(\alpha)^{\ast k}}} q^{\area(\pi, dr)} \Growpm_\ell(\hat{\pi}', \Sigma_\pi).
\end{multline}

\begin{prop}[{\cite[Proposition 8.4.7]{nsshuffle}}]\label{p flip LLT}
The operator $\barwb^\PP$ maps signed flagged row LLT polynomials to signed flagged column LLT polynomials. 
Precisely, for any $\pi \in \mathrm{D}_\ell(n)$ and any marking $\Sigma$ of $\pi$, we have:
\[
    \barwb^\PP  \Growpm_\ell(\pi, \Sigma)(X;t) = (-1)^n \Gcolpm_\ell(\pi, \Sigma)(X;t^{-1}).
\]
\end{prop}

Using Theorem \ref{newform} and Proposition \ref{p flip LLT}, we obtain an analogous version of Theorem \ref{q ns delta unmod}, which neatly expresses the identity in terms of signed flagged column LLT polynomials utilizing the operators $\overline{\tau_{u,\ell}^*}$ and $\umnab$.

\begin{cor}\label{cor:ns_delta_column}
With the above notation, we have:
\begin{multline}\label{e ns shuffle unmod 3}
	(-1)^{n+k} (-t)^{\ell-|\alpha|} \vartheta \left[\overline{\left(\tau_{qtu,\ell}^*\right)}^{-1} \left(\umnab\right)^{-1} \overline{\left(\tau_{qtu,\ell}^*\right)}^{-1} \right]_k \left( x_1^{\alpha_1} \dots x_\ell^{\alpha_\ell}\right) \\
	= \mathop{\sum_{(\pi, dr)\in \mathrm{D}(\alpha)^{\ast k}}} q^{\area(\pi, dr)} \Gcolpm_\ell(\hat{\pi}', \Sigma_\pi),
\end{multline}
where $\vartheta$ is the $\mathbb{Q}(q,t)$-antilinear map that sends \(f(X;q,t) \mapsto f(X;q^{-1},t^{-1})\).
\end{cor}

\begin{proof}
The result follows directly by applying the operator $(-1)^n \vartheta \barwb^{\PP}$ to both sides of equation \eqref{e ns shuffle unmod 2}, invoking Proposition \ref{p flip LLT}.
\end{proof}

Recall the definition of $\nsC_\alpha$ from \cite[Eq.~(189)]{nsshuffle}, which we restate here for convenience:
\begin{equation}\label{nsCdef}
    \nsC_\alpha := (-t)^{|\alpha|-\ell} \Pisf_\ell( x_1^{\alpha_1} \dots x_\ell^{\alpha_\ell} ) = (-t)^{|\alpha|-\ell} \pol \Big( \frac{x_1^{\alpha_1} \dots x_\ell^{\alpha_\ell}}{\prod_{1 \le i < j \le \ell} (1-t \spa x_i/x_j)} \Big) \in \P(\ell).
\end{equation}
These elements are related to the compositional Hall-Littlewood polynomials via Weyl symmetrization by \cite[Proposition 4.6]{BMPS19}:
\begin{equation}\label{weylnsc}
    \Weyl_1 \dots \Weyl_\ell \nsC_\alpha = C_{\alpha}(X;t^{-1}).
\end{equation}

\begin{defn}[{\cite[Section 8.4]{nsshuffle}}]
    Let \(\pi \in \mathrm{D}(\alpha)^{\ast k}\). A \emph{flagged word parking function} on $\pi$ is a labeling $w$ of the north steps of $\pi$ by positive integers that satisfies the following conditions:
    \begin{itemize}
        \item The labels are strictly increasing upward along each vertical run of $\pi$;
        \item The label of the first $j$ north steps lying on the main diagonal is $\le j$, for $1 \le j \le \ell$.
    \end{itemize}

We write $\FWPF_\alpha(\pi)$ for the set of all flagged word parking functions on $\pi$.

For any \(w \in \FWPF_\alpha(\pi)\), let $w_i$ denote the label of the north step in the $i$-th row. The \emph{diagonal inversion statistic} $\dinv(\pi,w)$ is the number of pairs of row indices $(i,j)$ with $1 \le i < j \le n$ such that $i$ attacks $j$ and $w_i < w_j$.
\end{defn}

This yields the culmination of our nonsymmetric operator formulations, corresponding to the identity \eqref{eq2} stated in the introduction:

\begin{cor}\label{cor:ns_delta_final}
With the above notation, we have:
    \begin{multline}\label{e ns shuffle unmod 4}
        (-1)^{n+k} \vartheta \left[\underline{\left(\tau_{qtu,\ell}^*\right)}^{-1} \modnab^{-1} \underline{\left(\tau_{qtu,\ell}^*\right)}^{-1}\right]_k \nsC_\alpha \\
        = \sum_{\pi \in \mathrm{D}(\alpha)^{\ast k}} \sum_{w \in \FWPF_\alpha(\pi)} q^{\area(\pi, dr)} t^{\dinv(\pi,w)} \xx^{\content(w)}.
    \end{multline}
\end{cor}

\begin{proof}
We apply the operator $\vartheta \Pisf_\ell \vartheta$ to both sides of \eqref{e ns shuffle unmod 3}. The computation proceeds as follows:
\begin{align*}
    & \mathrel{\phantom{=}} (-1)^{n+k} \vartheta \Pisf_\ell \vartheta (-t)^{\ell-|\alpha|} \vartheta \left[\overline{\left(\tau_{qtu,\ell}^*\right)}^{-1} 
    \left(\umnab\right)^{-1} \overline{\left(\tau_{qtu,\ell}^*\right)}^{-1} \right]_k
    \left( x_1^{\alpha_1} \dots x_\ell^{\alpha_\ell}\right) \\[1ex]
    &= (-1)^{n+k} \vartheta \Pisf_\ell \vartheta \vartheta (-t)^{|\alpha|-\ell}\left[\overline{\left(\tau_{qtu,\ell}^*\right)}^{-1} 
    \left(\umnab\right)^{-1} \overline{\left(\tau_{qtu,\ell}^*\right)}^{-1} \right]_k
    \left( x_1^{\alpha_1} \dots x_\ell^{\alpha_\ell}\right) \\[1ex]
    &= (-1)^{n+k} \vartheta \left[\Pisf_\ell \overline{\left(\tau_{qtu,\ell}^*\right)}^{-1} \Pisf_\ell^{-1} \Pisf_\ell
    \left(\umnab\right)^{-1} \Pisf_\ell^{-1} \Pisf_\ell \overline{\left(\tau_{qtu,\ell}^*\right)}^{-1} \Pisf_\ell^{-1}\right]_k (-t)^{|\alpha|-\ell} \Pisf_\ell \left( x_1^{\alpha_1} \dots x_\ell^{\alpha_\ell}\right) \\[1ex]
    &= (-1)^{n+k} \vartheta \left[\underline{\left(\tau_{qtu,\ell}^*\right)}^{-1} \modnab^{-1} \underline{\left(\tau_{qtu,\ell}^*\right)}^{-1}\right]_k (-t)^{|\alpha|-\ell}\Pisf_\ell \left( x_1^{\alpha_1} \dots x_\ell^{\alpha_\ell}\right) && \hspace{-10mm}\text{(by \eqref{eq831}, \eqref{defnpiptau})} \\[1ex]
    &= (-1)^{n+k} \vartheta \left[\underline{\left(\tau_{qtu,\ell}^*\right)}^{-1} \modnab^{-1} \underline{\left(\tau_{qtu,\ell}^*\right)}^{-1}\right]_k \nsC_\alpha && \text{(by \eqref{nsCdef})} \\[1ex]
    &= \mathop{\sum_{(\pi, dr)\in \mathrm{D}(\alpha)^{\ast k}}} q^{\area(\pi, dr)} \Gcol_\ell(\hat{\pi}', \Sigma_\pi) && \text{(by \eqref{eq5.4.1})} \\[1ex]
    &= \sum_{\pi \in \mathrm{D}(\alpha)^{\ast k}} \sum_{w \in \FWPF_\alpha(\pi)} q^{\area(\pi, dr)} t^{\dinv(\pi,w)} \xx^{\content(w)}. && \qedhere
\end{align*}
\end{proof}

\subsection{Connection with the symmetric version}
As mentioned in the introduction, we now demonstrate that Weyl symmetrization recovers the original compositional Delta theorem, 
conjectured by D'Adderio et al.\ \cite{theta} and proved by D'Adderio and Mellit \cite{proofdelta}.

\begin{prop}\label{restore}
Applying the Weyl symmetrization operator $\Weyl_1 \dots \Weyl_\ell$ to both sides of \eqref{e ns shuffle unmod 4} 
recovers the original compositional Delta theorem, 
which is exactly the statement of Theorem \ref{thm compositional delta conjecture}.
\end{prop}

Before proving Proposition \ref{restore}, we first establish the following preliminary lemmas.
\begin{lemma}\label{lemma783}
For any homogeneous element $f \in \P(\ell)$ of degree $d$, we have:
\[
    d_-^\ell \PP_\ell \Pisf_\ell^{-1} f = (-1)^d \omega \Weyl_1 \dots \Weyl_\ell f.
\]
\end{lemma}

\begin{proof}
By \eqref{pimac}, Blasiak et al.\ \cite[Proposition 7.8.3]{flaggedllt}, 
and the commutative diagram in \cite[Figure 2]{nsshuffle}, 
we obtain the following identity, which holds on the basis elements of $\P(\ell)$:
\[
    d_-^\ell \PP_\ell \Pisf_\ell^{-1} \tE_{\eta|\lambda} = d_-^\ell \PP_\ell \stE_{\eta|\lambda}
    = (-1)^d \Htild_\mu(X;q,t) = (-1)^d \omega \Weyl_1 \dots \Weyl_\ell \tE_{\eta|\lambda}.
\]
Since this identity holds for all basis elements of $\P(\ell)$, it naturally extends linearly to all of $\P(\ell)$, completing the proof.
\end{proof}

We also recall the following identity on $\sym[X]$ from \cite[Eq.~(193)]{nsshuffle}:
\begin{equation}\label{111}
    \vartheta \omega \left(\nabla' \right)^{-1} \omega \vartheta = \nabla'.
\end{equation}

\begin{prop}
We have the following identity on $\sym[X]$:
\begin{equation}\label{222}
    \vartheta \omega \underline{(qt)^i} \underline{h}_i^* \omega \vartheta = \underline{e}_i^*.
\end{equation}
The operator $\underline{(qt)^i}$ here denotes multiplication by $(qt)^i$.
\end{prop}

\begin{proof}
Recall that the modified Macdonald polynomials satisfy the fundamental property:
\[
    \omega \Htild_\mu(X;q,t) = T_\mu \Htild_\mu(X;q^{-1},t^{-1}).
\]
Using this, we compute the action of the operator on the Macdonald basis as follows:
    \begin{align*}
        & \mathrel{\phantom{=}} \vartheta \omega \underline{(qt)^i} \underline{h}_i^* \omega \vartheta \Htild_\mu(X;q,t) \\[1ex]
        &= \vartheta \omega \underline{(qt)^i} \underline{h}_i^* \omega \Htild_\mu(X;q^{-1},t^{-1}) \\[1ex]
        &= \vartheta \omega \underline{(qt)^i} \underline{h}_i^* \left(T_\mu\right)^{-1} \Htild_\mu(X;q,t) \\[1ex]
        &= \vartheta \omega (qt)^i h_i\left[\cfrac{X}{M}\right] \left(T_\mu\right)^{-1} \Htild_\mu(X;q,t) \\[1ex]
        &= \vartheta (qt)^i e_i\left[\cfrac{X}{M}\right] \left(T_\mu\right)^{-1} T_\mu \Htild_\mu(X;q^{-1},t^{-1}) \\[1ex]
        &= \left(\cfrac{1}{qt}\right)^i e_i\left[\cfrac{qtX}{M}\right] \Htild_\mu(X;q,t) = e_i\left[\cfrac{X}{M}\right] \Htild_\mu(X;q,t) = \underline{e}_i^* \Htild_\mu(X;q,t).
    \end{align*}
Since this action maps basis elements to basis elements consistently, the identity follows for all symmetric functions.
\end{proof}

\begin{proof}[Proof of Proposition \ref{restore}]
We first apply the Weyl symmetrization operator $\Weyl_1 \dots \Weyl_\ell$ to the left-hand side of 
\eqref{e ns shuffle unmod 4}. The computation proceeds as follows:
    \begin{align}
        & \mathrel{\phantom{=}} (-1)^{n+k} \Weyl_1 \dots \Weyl_\ell \vartheta \left[\underline{\left(\tau_{qtu,\ell}^*\right)}^{-1} \modnab^{-1} \underline{\left(\tau_{qtu,\ell}^*\right)}^{-1}\right]_k \nsC_\alpha \nonumber \\[1ex]
        &= (-1)^{n+k} \vartheta \Weyl_1 \dots \Weyl_\ell \left[\Pisf_\ell \overline{\left(\tau_{qtu,\ell}^*\right)}^{-1} \umnab^{-1} \overline{\left(\tau_{qtu,\ell}^*\right)}^{-1} \Pisf_\ell^{-1}\right]_k \nsC_\alpha && \text{(by \eqref{eq831}, \eqref{defnpiptau})} \nonumber \\[1ex]
        &= (-1)^{n+k} \vartheta \Weyl_1 \dots \Weyl_\ell \Pisf_\ell \PP_\ell^{-1} \left[\left(\tau_{qtu,\ell}^*\right)^{-1} \Mnab^{-1} \left(\tau_{qtu,\ell}^*\right)^{-1} \PP_\ell \Pisf_\ell^{-1}\right]_k \nsC_\alpha && \text{(by \eqref{eq831}, \eqref{defnptau})} \nonumber \\[1ex]
        &= (-1)^{n+k} \vartheta \Pisf_0 \PP_0^{-1} d_-^\ell \PP_\ell \PP_\ell^{-1} \left[\left(\tau_{qtu,\ell}^*\right)^{-1} \Mnab^{-1} \left(\tau_{qtu,\ell}^*\right)^{-1} \PP_\ell \Pisf_\ell^{-1}\right]_k \nsC_\alpha && \text{(by Lemma \ref{useful})} \nonumber \\[1ex]
        &= (-1)^{n+k} \vartheta (-1)^n \omega d_-^\ell \left[\left(\tau_{qtu,\ell}^*\right)^{-1} \Mnab^{-1} \left(\tau_{qtu,\ell}^*\right)^{-1} \PP_\ell \Pisf_\ell^{-1}\right]_k \nsC_\alpha && \text{(by \eqref{pi0p0})} \nonumber \\[1ex]
        &= (-1)^{n+k} \vartheta (-1)^n \omega \left[\left(\tau_{qtu}^*\right)^{-1} \left(\nabla' \right)^{-1} \left(\tau_{qtu}^*\right)^{-1}\right]_k d_-^\ell \PP_\ell \Pisf_\ell^{-1} \nsC_\alpha && \text{(by \eqref{eqtau-1}, \eqref{mnab-1})} \nonumber \\[1ex]
        &= (-1)^{n+k} \vartheta (-1)^n \omega \left[\left(\tau_{qtu}^*\right)^{-1} \left(\nabla' \right)^{-1} \left(\tau_{qtu}^*\right)^{-1}\right]_k (-1)^{n-k} \omega \Weyl_1 \dots \Weyl_\ell \nsC_\alpha && \text{(by Lemma \ref{lemma783})} \nonumber \\[1ex]
        &= (-1)^{n+k} \vartheta (-1)^n \omega \left[\left(\tau_{qtu}^*\right)^{-1} \left(\nabla' \right)^{-1} \left(\tau_{qtu}^*\right)^{-1}\right]_k (-1)^{n-k} \omega C_\alpha(X;t^{-1}) && \text{(by \eqref{weylnsc})} \nonumber \\[1ex]
        &= (-1)^n \vartheta \omega \left[\left(\tau_{qtu}^*\right)^{-1} \left(\nabla' \right)^{-1} \left(\tau_{qtu}^*\right)^{-1}\right]_k \omega \vartheta C_\alpha(X;t) \nonumber \\[1ex]
        &= \left[(-1)^n \vartheta \omega \left(\tau_{qtu}^*\right)^{-1} \left(\nabla' \right)^{-1} \left(\tau_{qtu}^*\right)^{-1} \omega \vartheta \right]_k C_\alpha(X;t). \nonumber \\[-1ex]
        & \label{e ns shuffle unmod 4 left}
    \end{align}

Next, we extract the coefficient of $(-u)^k$ in the operator 
$(-1)^n \vartheta \omega \left(\tau_{qtu}^*\right)^{-1} \left(\nabla' \right)^{-1} \left(\tau_{qtu}^*\right)^{-1} \omega \vartheta$, 
which corresponds to the notation $[\cdot]_k$. This gives us:
    \begin{align*}
        & \mathrel{\phantom{=}} (-1)^n \vartheta \omega (-qt)^k \sum_{i=0}^k \underline{h}_i^* \left(\nabla' \right)^{-1} \underline{h}_{k-i}^* \omega \vartheta \\
        &= (-1)^{|\alpha|} \sum_{i=0}^k \vartheta \omega \underline{(qt)^i} \underline{h}_{i}^* \left(\nabla' \right)^{-1} \underline{(qt)^{k-i}} \underline{h}_{k-i}^* \omega \vartheta.
    \end{align*}

Combining this with identities \eqref{111} and \eqref{222}, 
we continue the computation of \eqref{e ns shuffle unmod 4 left} to obtain:
\begin{align*}
    & \mathrel{\phantom{=}} (-1)^{|\alpha|} \sum_{i=0}^k \vartheta \omega \underline{(qt)^i} \underline{h}_i^* \left(\nabla' \right)^{-1} \underline{(qt)^{k-i}} \underline{h}_{k-i}^* \omega \vartheta C_\alpha(X;t) \\[1ex]
    &= (-1)^{|\alpha|} \sum_{i=0}^k \vartheta \omega \underline{(qt)^i} \underline{h}_i^* \omega \vartheta 
    \vartheta \omega \left(\nabla' \right)^{-1} \omega \vartheta 
    \vartheta \omega \underline{(qt)^{k-i}} \underline{h}_{k-i}^* \omega \vartheta C_\alpha(X;t) && \text{(since $\omega\vartheta\vartheta\omega = \mathrm{id}$)} \\[1ex]
    &= (-1)^{|\alpha|} \sum_{i=0}^k \underline{e}_i^* \nabla' \underline{e}_{k-i}^* C_\alpha(X;t) && \text{(by \eqref{111}, \eqref{222})} \\[1ex]
    &= \Theta_{e_k} \nabla C_\alpha(X;t). && \text{(by \eqref{remarksign})} \\[1ex]
\end{align*}

Next, we apply the Weyl symmetrization operator to the right-hand side of \eqref{e ns shuffle unmod 4}, which yields:
\begin{align}
    & \mathrel{\phantom{=}} \Weyl_1 \dots \Weyl_\ell \sum_{\pi \in \mathrm{D}(\alpha)^{\ast k}} 
    \sum_{w \in \FWPF_\alpha(\pi)} q^{\area(\pi, dr)} t^{\dinv(\pi,w)} \xx^{\content(w)} \nonumber \\[1ex]
    &= \mathop{\sum_{(\pi, dr)\in \mathrm{D}(\alpha)^{\ast k}}} q^{\area(\pi, dr)} \Weyl_1 \dots \Weyl_\ell \Gcol_\ell(\hat{\pi}', \Sigma_\pi) \nonumber \\[1ex]
    &= \mathop{\sum_{(\pi, dr)\in \mathrm{D}(\alpha)^{\ast k}}} q^{\area(\pi, dr)} d_-^\ell \chi_\ell(\hat{\pi}', \Sigma_\pi) && \text{(by \eqref{weylcolumn})} \nonumber \\[1ex]
    &= d_-^\ell \mathop{\sum_{(\pi, dr)\in \mathrm{D}(\alpha)^{\ast k}}} q^{\area(\pi, dr)} \chi_\ell(\hat{\pi}', \Sigma_\pi) \nonumber \\[1ex]
    &= \mathop{\sum_{P\in \mathrm{LD}(n)^{\ast k}}}_{\dcomp(P)=\alpha} q^{\area(P)} t^{\dinv(P)} \xx^{P}. && \text{(by \cite[Remark 6.21]{theta})} \label{e ns shuffle unmod 4 right}
\end{align}

Equating the fully expanded left-hand side and right-hand side \eqref{e ns shuffle unmod 4 left} and \eqref{e ns shuffle unmod 4 right}, we immediately deduce:
\[
    \Theta_{e_k} \nabla C_{\alpha}(X;t) = \mathop{\sum_{P\in \LD(n)^{\ast k}}}_{\dcomp(P)=\alpha} q^{\area(P)} t^{\dinv(P)} \xx^{P},
\]
which rigorously completes the proof of the proposition.
\end{proof}

\section{Further remarks}
\subsection{Atom positivity conjectures}
Drawing inspiration from the works of Blasiak et al.\ \cite{nsshuffle, flaggedllt}, we propose analogous stable atom positivity conjectures in our framework.

Recall the Demazure operator $\pi_i$ defined in \eqref{dema}. For any vector \(\alpha \in \ZZ^N\), 
we define the corresponding Demazure atom as follows:
\begin{equation}\label{atom}
    \Acal_\alpha := \hat{\pi}_{w} (\xx^{\alpha_+}),
\end{equation}
where \(\hat{\pi}_i := \pi_i - 1\), and $\pi_w$ denotes the shortest permutation such that \(w(\alpha_+) = \alpha\), with $\alpha_+$ being the partition obtained by sorting the parts of $\alpha$ in weakly decreasing order. Here $x^{\alpha_+}$ denotes the monomial $x_1^{(\alpha_+)_1} \dots x_N^{(\alpha_+)_N}$.

The stable atoms are then defined by:
\begin{equation}\label{stableatom}
    \Acal_{\eta| \lambda} = \Acal_{\eta| \lambda}(X) := \Weyl_{\ell+1} \dots \Weyl_{\ell+k} \Acal_{(\eta; \lambda)}(X_{\ell+k}),
\end{equation}
where \((\eta | \lambda) \in \pairsr\), and \(k := \ell(\lambda)\) denotes the length of the partition $\lambda$.
Notably, as $(\eta | \lambda)$ ranges over $\pairsr$, the elements $\Acal_{\eta| \lambda}$ form a basis for the space $\P(\ell)$.

Furthermore, by \cite[Lemma 7.8.2]{flaggedllt}, they satisfy the following structural property:
\[
    \Weyl_1 \dots \Weyl_\ell \Acal_{\eta| \lambda}(X) =
    \begin{cases}
        s_{(\eta; \lambda)}(X) &  \text{if $(\eta; \lambda)$ is weakly decreasing}, \\
        0           &      \text{otherwise}.
    \end{cases}
\]

As a consequence, using \cite[Eq.~(201)]{nsshuffle}, for any element \(f \in \P(\ell)\), we have the implication:
\begin{equation}\label{201}
    \text{$f$ is stable atom positive} \implies \text{$\Weyl_1 \dots \Weyl_\ell f$ is Schur positive}.
\end{equation}
This demonstrates that stable atom positivity serves as a strictly stronger condition than Schur positivity in the symmetric function setting.

\begin{conj}
    We conjecture the following positivity phenomena:
\begin{enumerate}
    \item Flagged row and column LLT polynomials are stable atom positive.
    \item The quantity $\Pisf_\ell \spa \alg$ appearing in \eqref{eq1} is stable atom positive.
    \item The quantity \((-1)^{n+k} \left[\underline{\tau_{qtu,\ell}^*}^{-1} \modnab^{-1} \underline{\tau_{qtu,\ell}^*}^{-1}\right]_k \nsC_\alpha\) 
    appearing in \eqref{eq2} is stable atom positive.
\end{enumerate}
\end{conj}

The first conjecture was originally proposed by Blasiak et al.\ in \cite[Conjecture 8.5.3]{nsshuffle}. 
It is straightforward to verify that the first conjecture logically implies the latter two.

\subsection{Nonsymmetric \texorpdfstring{$\Theta$}{Theta} operator}
Analogous to the definitions of the signed nonsymmetric nabla operator $\umnab$ and 
the standard nonsymmetric nabla operator $\modnab$ given by Blasiak et al.\ \cite{nsshuffle} (see Eqs.\ \eqref{141} and \eqref{142}), 
we introduce the corresponding nonsymmetric $\Theta$ operator.

Following D'Adderio et al.\ \cite[Eq.~(27)]{theta}, we define two operators $\underline{\mathbf{\Pi}}$ and $\overline{\mathbf{\Pi}}$ 
acting on the space $\P(\ell)$ as follows:
\begin{equation}
    \overline{\mathbf{\Pi}} \stE_{\eta| \lambda}(X ;q,t) := \Pi_\mu \stE_{\eta| \lambda}(X ;q,t), \qquad \text{where } \mu = (\eta; \lambda)_+,
\end{equation}
\begin{equation}
    \underline{\mathbf{\Pi}} \tE_{\eta| \lambda}(X ;q,t) := \Pi_\mu \tE_{\eta| \lambda}(X ;q,t), \qquad \text{where } \mu = (\eta; \lambda)_+.
\end{equation}

By the identity \eqref{pimac}, these two operators are related via the following conjugation:
\[
    \Pisf_\ell \overline{\mathbf{\Pi}} \Pisf^{-1}_\ell = \underline{\mathbf{\Pi}}.
\]

Using these two building blocks, we can now define the nonsymmetric $\Theta$ operators $\umthe_f$ and $\modthe_f$, 
for any \(f \in \sym[X]\) and \(F \in \P(\ell)\):
\begin{equation} \label{umthe}
    \umthe_f F  :=  
    \begin{cases}
        \overline{\mathbf{\Pi}} f^* \overline{\mathbf{\Pi}}^{-1} F  & \text{if } n \ge 1,\\
        0  & \text{if } n=0 \text{ and } k \ge 1, \\
        f \cdot F  & \text{if } n=0 \text{ and } k=0.
    \end{cases}
\end{equation}
\begin{equation} \label{modthe}
    \modthe_f F  :=  
    \begin{cases}
        \underline{\mathbf{\Pi}} f^* \underline{\mathbf{\Pi}}^{-1} F  & \text{if } n \ge 1,\\
        0  & \text{if } n=0 \text{ and } k \ge 1, \\
        f \cdot F  & \text{if } n=0 \text{ and } k=0.
    \end{cases}
\end{equation}

In the symmetric setting, D'Adderio and Mellit \cite[Proposition 6.2]{proofdelta} established the following identity, 
which we stated earlier in \eqref{thetanabla}:
\[
    \Theta_{e_k} \nabla' = \sum_{i=0}^k \e_i^* \nabla' \e_{k-i}^*.
\]

\begin{open}
Is there an analogous structural identity in the nonsymmetric setting of $\P(\ell)$, comparable to \eqref{thetanabla}, 
that involves the operators $\umnab$ and $\umthe_f$, or $\modnab$ and $\modthe_f$? 
\end{open}

Such an identity would provide a highly compact and beautiful reformulation for our main results \eqref{eq1} and \eqref{eq2}.


\addresseshere

\end{document}